\documentclass[a4paper,10pt,leqno]{amsart}
        \title[Topological $K$-(co-)homology of classifying spaces \ldots]
       {Topological $K$-(co-)homology of classifying spaces of discrete groups}
       \author{Michael Joachim}
       \author{Wolfgang L\"uck}
              \address{Westf\"alische Wilhelms-Universit\"at M\"unster\\
               Mathematisches Institut\\
               Einsteinstr.~62,
               D-48149 M\"unster, Germany}
       \address{Rheinische Wilhelms-Universit\"at Bonn\\
               Mathematisches Institut\\
               Endenicher Allee 62, 53115 Bonn, Germany}
              \email{joachim@math.uni-muenster.de}
      \urladdr{http://www.math.uni-muenster.de/u/joachim/}
        \email{wolfgang.lueck@him.uni-bonn.de}
      \urladdr{http://www.him.uni-bonn.de/lueck/}
              \date{January  2012}

\usepackage{pdfsync}
\usepackage{hyperref}
\usepackage{calc}
\usepackage{enumerate,amssymb}
\usepackage{color}
\usepackage[all]{xy}
\usepackage{tikz}
\usepackage{pdfcolmk}

\DeclareMathAlphabet\EuR{U}{eur}{m}{n}
\SetMathAlphabet\EuR{bold}{U}{eur}{b}{n}
\newcommand{\version}[1]                    
{\begin{center} last edited on #1\\
last compiled on \today\\
name of texfile: \jobname
\end{center}
}

\newtheorem{theorem}{Theorem}[section]
\newtheorem{lemma}[theorem]{Lemma}

\newtheorem{definition}[theorem]{Definition}
\newtheorem{example}[theorem]{Example}
\newtheorem{remark}[theorem]{Remark}
\newtheorem{corollary}[theorem]{Corollary}

{\catcode`@=11\global\let\c@equation=\c@theorem}

\newcommand{\forget}[1]{}

\newcommand{\curs}{\EuR}

\newcommand{\Sub}{\curs{Sub}}

\newcommand{\N}{{\mathcal N}}

\newcommand{\calfin}{{\mathcal F}\!{\mathcal I}\!{\mathcal N}}
\newcommand{\calmfin}{{\mathcal M}\!{\mathcal F}\!{\mathcal I}\!{\mathcal N}}
\newcommand{\calf}{{\mathcal F}}

\newcommand{\calh}{{\mathcal H}}
\newcommand{\calp}{{\mathcal P}}

\newcommand{\higherlim}[3]{\underleftarrow{\lim}_{#1}^{#2}#3}
\newcommand{\highercolim}[3]{\underrightarrow{\dirlim}_{#1}^{#2}#3}
\newcommand{\invlim}[2]{\higherlim{#1}{}{#2}}
\newcommand{\colim}[2]{\highercolim{#1}{}{#2}}

\newcommand{\pt}{\{\bullet\}}

\newcommand{\IC}{{\mathbb C}}

\newcommand{\IF}{{\mathbb F}}

\newcommand{\IK}{{\mathbb K}}

\newcommand{\IN}{{\mathbb N}}

\newcommand{\IQ}{{\mathbb Q}}
\newcommand{\IR}{{\mathbb R}}

\newcommand{\IZ}{{\mathbb Z}}

\newcommand{\bfE}{\ensuremath{\mathbf{E}}}

\newcommand{\fulldot}{{\circ}}

\newcommand{\cok}{\operatorname{coker}}
\newcommand{\cts}{\operatorname{cts}}

\newcommand{\image}{\operatorname{image}}

\newcommand{\con}{\operatorname{con}}
\newcommand{\cyl}{\operatorname{cyl}}
\newcommand{\dirlim}{\operatorname{colim}}

\newcommand{\ext}{\operatorname{ext}}
\newcommand{\id}{\operatorname{id}}
\newcommand{\im}{\operatorname{im}}

\newcommand{\pr}{\operatorname{pr}}
\newcommand{\res}{\operatorname{res}}

\newcommand{\AI}{{\mathbb I}}

\newcounter{commentcounter}

\newcommand{\xycomsquare}[8]                   
{\xymatrix
{#1 \ar[r]^{#2} \ar[d]^{#4} &
#3 \ar[d]^{#5}  \\
#6\ar[r]^{#7} &
#8
}
}

\newcommand{\xycomsquareminus}[8]                      
{\xymatrix{#1 \ar[r]^-{#2} \ar[d]^-{#4} &
#3 \ar[d]^-{#5}  \\
#6\ar[r]^-{#7} &
#8
}
}

\begin{document}

\maketitle

\typeout{-----------------------  Abstract  ------------------------}
\begin{abstract}
Let $G$ be a discrete group. We give methods to compute for a generalized
(co-)homology theory its values on the Borel construction $EG \times_G X$
of a proper $G$-$CW$-complex $X$  satisfying certain finiteness
conditions. In particular we give  formulas computing the topological $K$-(co)homology
$K_*(BG)$ and $K^*(BG)$ up to finite abelian torsion groups.
They apply  for instance to arithmetic groups,
word hyperbolic groups,  mapping class groups
and discrete cocompact subgroups of almost connected Lie groups.
For finite groups $G$ these formulas are sharp.
The main new tools we use for the $K$-theory calculation are a  Cocompletion Theorem and 
Equivariant Universal Coefficient Theorems which are of independent interest. In the case where
$G$ is a finite group these theorems reduce to well-known results of Greenlees and B\"okstedt.
\end{abstract}

\typeout{-----------------------  Introduction ------------------------}

\setcounter{section}{-1}
\section{Introduction}
\label{sec:Introduction}

One of our main goals in this paper is to compute for a discrete group $G$, a proper $G$-$CW$-complex $X$
and a generalized cohomology theory $\calh^*$  and a generalized homology theory
$\calh_*$ the groups $\calh^*(EG \times_G X)$ and  $\calh_*(EG \times_G X)$.
In particular the case is interesting, where $X$ can be chosen to be non-equivariantly contractible because then
$EG \times_G X$ is a model for the classifying space
$BG$. The main results are Theorem~\ref{the:computing_calh_upper_ast(EG_times_G_X)_for_proper_X}
and Theorem~\ref{the:computing_calh_lower_ast(EG_times_G_X)_for_proper_X}.
In the introduction we will concentrate on the case, where $\calh^*$ and $\calh_*$ are
topological $K$-theory $K^*$ and $K_*$ and on classifying spaces $BG$.

Throughout the paper $H_*(Y;M)$ and $H^*(Y;M)$
denote singular (co-)-homology of $Y$ with coefficients in the abelian group $M$, and we omit
$M$ from the notation in the case $M = \IZ$.
Let  $\IZ\widehat{_p}$ be the ring of $p$-adic integers, which is the inverse limit
of the inverse system of projections $\IZ/p \xleftarrow{\pr_2}\IZ/p^2 \xleftarrow{\pr_3}
\IZ/p^3 \xleftarrow{\pr_4} \ldots$. Denote by
$\IZ/p^{\infty}$ the quotient $\IZ[1/p]/\IZ$ which
is isomorphic to the colimit of the directed system of inclusions
$\IZ/p \xrightarrow{p} \IZ/p^2\xrightarrow{p} \IZ/p^3 \xrightarrow{p}
\ldots$.  The proof of the following result
will be given in Section~\ref{sec:Universal_coefficient_theorems_for_K-theory}.

\begin{theorem}[Topological $K$-theory of classifying spaces]
\label{the:intro}
Let $G$ be a discrete group. Let $X$ be a
finite proper $G$-$CW$-complex.
Suppose for every $k \in \IZ$ that $\widetilde{H}_k(X) = 0$  vanishes.
Given a prime number $p$ and $k \in \IZ$, define the  natural number
\begin{eqnarray*}
r_p^k(G) & := &  \sum_{(g) \in \con_p(G)}  \;\sum_{i \in \IZ}
\dim_{\IQ}\left(H^{k+2i}(BC_G\langle g \rangle;\IQ)\right),
\end{eqnarray*}
where $\con_p(G)$ denotes the set of conjugacy classes of non-trivial elements of
$p$-power order and $C_G\langle g \rangle$ is the centralizer in $G$ of
the cyclic subgroup $\langle g \rangle$ generated by $g$.
Let $\calp(G)$ be the set of primes $p$ which divide the order $H$ of some finite subgroup
$H\subseteq G$. Then:

\begin{enumerate}

\item \label{the:intro:K_upper_ast}
There is an exact sequence
\begin{multline*}
0 \to A \to K^k(G\backslash X) \to K^k(BG)
\to
B \times \prod_{p \in \calp(G)} \left(\IZ\widehat{_p}\right)^{r_p^k(G)} \to C \to 0,
\end{multline*}
where $A$, $B$ and $C$ are finite abelian groups with
\[
A \otimes_{\IZ} \IZ\left[\frac{1}{\calp(G)}\right] =
B \otimes_{\IZ}  \IZ\left[\frac{1}{\calp(G)}\right] =
C \otimes_{\IZ}  \IZ\left[\frac{1}{\calp(G)}\right] =0;
\]

\item \label{the:intro:K_lower_ast}
Dually there is an exact sequence
\[
0 \to C' \to \coprod_{p \in \calp(G)} (\IZ/{p^\infty})^{r_p^{k+1}(G)} \times B'
 \to K_k(BG) \to K_k(G\backslash X)
\to A' \to 0,
\]
where $A'$, $B'$ and $C'$ are finite abelian groups with
\[
A' \otimes_{\IZ} \IZ\left[\frac{1}{\calp(G)}\right] =
B' \otimes_{\IZ}  \IZ\left[\frac{1}{\calp(G)}\right] =
C' \otimes_{\IZ}  \IZ\left[\frac{1}{\calp(G)}\right] =0;
\]

\item \label{the:intro:rationally}
If we invert all primes in $\calp(G)$, then we obtain isomorphisms
\begin{eqnarray*}
K^k(BG) \otimes_{\IZ} \IZ\left[\frac{1}{\calp(G)}\right] & \cong &
K^k(G\backslash X) \otimes_{\IZ} \IZ\left[\frac{1}{\calp(G)}\right]
\times \prod_{p \in \calp(G)} \left(\IQ\widehat{_p}\right)^{r_p^k(\underline{E}G)};
\\
K_k(BG)\otimes_{\IZ} \IZ\left[\frac{1}{\calp(G)}\right] & \cong &
K_k(G\backslash X) \otimes_{\IZ} \IZ\left[\frac{1}{\calp(G)}\right].
\end{eqnarray*}
\end{enumerate}
\end{theorem}

Under the conditions appearing in Theorem~\ref{the:intro} the sets $\con_p(G)$ and $\calp(G)$ are finite
and the dimension of $BC_G\langle g \rangle$ is bounded by the finite dimension of $X$.
Hence the numbers $r_p^k(X)$ are well-defined.

The exact sequences of assertions~\ref{the:intro:K_upper_ast}
and~\ref{the:intro:K_lower_ast} of Theorem~\ref{the:intro} are in fact dual to each other when
working in the category of topological groups. Namely, recall that
skeletal filtration on $BG$ imposes the structure of a pro-discrete
group on $K^*(BG)$, while the $p$-adic integers can be equipped with
the pro-finite topology. With these topologies the exact sequence of
Theorem~\ref{the:intro}~\ref{the:intro:K_lower_ast} is the
Pontryagin dual of the exact sequence of Theorem~\ref{the:intro}~\ref{the:intro:K_upper_ast}
 and vice versa. Moreover both exact sequence of
 Theorem~\ref{the:intro} are also exact sequences in the category of topological groups
 in the sense of~\cite{Yoneda(1960)}.

A model $\underline{E}G$ for the \emph{classifying space for proper
$G$-actions} is a proper $G$-$CW$-complex, whose $H$-fixed point
sets are contractible for all finite subgroups $H \subseteq G$. It
is unique up to $G$-homotopy. It is a good candidate for $X$ in
Theorem~\ref{the:intro}, provided that there exists a finite
$G$-$CW$-complex model for $\underline{E}G$. Examples, for which
this is true, are arithmetic groups in a semisimple connected linear
$\IQ$-algebraic group~\cite{Borel-Serre(1973)},~\cite{Serre(1979)},
mapping class groups (see~\cite{Mislin(2010)}), groups which are
hyperbolic in the sense of Gromov
(see\cite{Meintrup(2000)},~\cite{Meintrup-Schick(2002)}), virtually
poly-cyclic groups, and groups which are cocompact discrete
subgroups of Lie groups with finitely many path
components~\cite[Corollary~4.14]{Abels(1978)}. On the other hand,
for any $CW$-complex $Y$ there exists a group $G$ such that $Y$ and
$G\backslash \underline{E}G$ are homotopy equivalent
(see~\cite{Leary-Nucinkis(2001a)}). More information about these
spaces $\underline{E}G$ can be found for instance
in~\cite{Baum-Connes-Higson(1994)},~\cite{Lueck(2005s)},~\cite[Section~I.6]{Dieck(1987)}.

In order to apply Theorem~\ref{the:intro} one needs to understand
the $CW$-complex $G\backslash X$. This is often possible
using geometric input, in particular in the case $X = \underline{E}G$ for the groups
mentioned above. Notice that $q$-torsion in
$K^k(BG)$ and $K_k(BG)$ for a prime number $q$ which does not belong to
$\calp(G)$ must come from the $q$-torsion in  $K^k(G\backslash X)$ and $K_k(G\backslash X)$.

The rational version of our formula for $K$-cohomology has already been proved using equivariant Chern characters
in~\cite[Theorem~0.1]{Lueck(2007)}, see also~\cite[Theorem 6.3]{Adem(1993b)}.

If $G$ is finite, a model for $\underline{E}G$ is $\pt$ and one gets a complete answer integrally,
see for instance~\cite[Theorem~0.3]{Lueck(2007)}.

We will recall the Completion Theorem~\ref{the:Completion_Theorem}
of~\cite[Theorem~6.5]{Lueck-Oliver(2001b)}) and deduce from it in
Section~\ref{sec:Completion_and_cocompletion_theorems}.

\begin{theorem}[Cocompletion Theorem]
\label{the:Cocompletion_Theorem}
Let $G$ be a discrete group. Let $X$ be a finite proper
$G$-$CW$-complex and let $L$ be a finite dimensional
proper $G$-$CW$-complex whose isotropy subgroups have bounded order.
Fix a $G$-map $f \colon X \to L$ and regard $K^*_G(X)$ as a module over
$\IK_G(L)$. Moreover, let $I=\AI_G(L)$ be the augmentation ideal (see
Definition~\ref{def:augmentation_ideal}).

 Then there is a short exact sequence
\begin{multline*}
\lefteqn{0\to \colim{n \ge 1}{\ext^1_{\IZ}(K^{*+1}_G(X)/I^n\cdot K^{*+1}_G(X), \IZ )}
\to K_*(EG\times_G X)}\\
\to
 \colim{n \ge 1}{\hom_{\IZ}(K^*_G(X)/I^n\cdot K^*_G(X) , \IZ)} \to 0.
\end{multline*}
\end{theorem}
When working in the category of topological abelian groups and
continuous homomorphisms the sequence can be written in the
following more compact form
\[
0\to {\ext^1_{\cts}(K^{*+1}_G(X)\widehat{_I}, \IZ )} \to
K_*(EG\times_G X) \to {\hom_{\cts}(K^*_G(X)\widehat{_I}, \IZ)} \to 0.
\]
Theorem~\ref{the:Cocompletion_Theorem} is closely related to the local
cohomology approach to equivariant $K$-homology of
Greenlees~\cite{Greenlees(1993a)} (see
Remark~\ref{rem:Cocompletion_Theorem_and_the_approach_of_Greenlees}
below).

In Section~\ref{sec:Universal_coefficient_theorems_for_K-theory} we
prove

\begin{theorem}[Equivariant Universal Coefficient Theorem for $K$-theory]
\label{the:Equivariant_Universal_Coefficient_Theorem_for_K-theory}
Let $G$ be a discrete group. Let $X$ be a finite proper $G$-$CW$-complex $X$.

Then there are short exact sequences, natural in $X$,
\begin{eqnarray}
\label{UCT1}
0 \to \ext_\IZ(K_{*-1}^G(X),\IZ) \to K^*_G(X) \to  \hom_\IZ(K_*^G(X),\IZ)\to 0;\\
\label{UCT2}
0 \to \ext_\IZ(K^{*+1}_G(X),\IZ) \to K_*^G(X) \to  \hom_\IZ(K^*_G(X),\IZ)\to 0,
\end{eqnarray}
where the homomorphism on the right hand sides are given by~\eqref{adjoint.before.UCT.1}
and~\eqref{adjoint.before.UCT.2} respectively. The sequence splits unnaturally.
\end{theorem}

Theorem~\ref{the:Equivariant_Universal_Coefficient_Theorem_for_K-theory}  reduces for
finite groups to the corresponding results of  B\"okstedt~\cite{Boekstedt(1981)}
as in explained in Remark~\ref{rem:Boekstedt_universal_coefficient_theorem}.





The work was financially supported by Sonderforschungsbereich 878 \emph{Groups,
  Geometry and Actions} in M\"unster,  and the Leibniz-Preis of the second author.

\vspace{6mm}


\typeout{-----------------------  Section 1  ------------------------}

\section{Some preliminaries about pro-modules}
\label{sec:Some_preliminaries_about_pro-modules}

It will be crucial to handle pro-systems and pro-isomorphisms and not
to pass directly to inverse limits. Otherwise we would loose
important information which is for instance needed  in order to pass from $K$-cohomology to
$K$-homology using universal coefficients theorems.
In this section we fix our notation for handling pro-$R$-modules for a
commutative ring $R$, where ring always means associative ring with unit.
For the definitions in full generality see for instance~\cite[Appendix]{Artin-Mazur(1969)}
or~\cite[Section~2]{Atiyah-Segal(1969)}. This exposition agrees with the one
in~\cite[Section~2]{Lueck(2007)} and is repeated for the reader's convenience.

For simplicity, all pro-$R$-modules dealt
with here will be indexed by the positive  integers.  We
write $\{M_n,\alpha_n\}$ or briefly $\{M_n\}$ for the inverse system
\[
 M_0 \xleftarrow{\alpha_1} M_1 \xleftarrow{\alpha_2} M_2 \xleftarrow{\alpha_3}
M_3 \xleftarrow{\alpha_4} \cdots.
\]
and also write
$\alpha_n^m := \alpha_{m+1} \circ \cdots \circ \alpha_{n}\colon G_n \to G_m$
for $n > m$ and put $\alpha^n_n =\id_{G_n}$.  For the purposes here, it will
suffice (and greatly simplify the notation)
to work with ``strict'' pro-homomorphisms
$\{f_n\} \colon \{M_n,\alpha_n\} \to \{N_n,\beta_n\}$, i.e.,
a collection of  homomorphisms $f_n \colon M_n \to N_n$
for $n \ge 1$ such that  $\beta_{n}\circ f_n = f_{n-1}\circ\alpha_{n}$ holds
for each $ n\ge 2$.  Kernels  and cokernels of strict homomorphisms are
defined in the obvious way.

A pro-$R$-module $\{M_n,\alpha_n\}$  will be called
\emph{pro-trivial} if for each $m \ge 1$, there
is some $n\ge m$ such that  $\alpha_n^m = 0$.  A strict homomorphism
$f\colon \{M_n,\alpha_n\} \to \{N_n,\beta_n\}$ is a
\emph{pro-isomorphism}
if and only if $\ker(f)$ and $\cok(f)$ are both
pro-trivial, or, equivalently, for each $m\ge 1$ there is some
$n\ge m$ such that $\im(\beta_n^m) \subseteq \im(f_m)$
and $\ker(f_n) \subseteq \ker(\alpha_n^m)$.
A sequence of strict homomorphisms
\[
\{M_n,\alpha_n\} \xrightarrow{\{f_n\}} \{M_n',\alpha_n'\}
\xrightarrow{g_n} \{M_n'',\alpha_n''\}
\]
will be called \emph{exact} if the sequences of $R$-modules
$M_n \xrightarrow{f_n} N_n \xrightarrow{g_n} M_n''$ is
exact for each $n \ge 1$, and it is called \emph{pro-exact} if
$g_n \circ f_n = 0$  holds for $n \ge 1$ and
the pro-$R$-module  $\{\ker(g_n)/\im(f_n)\bigr\}$ is pro-trivial.

The following results will be needed later.

\begin{lemma} \label{lem:pro-exactness_and_limits}
Let $0 \to \{M_n',\alpha_n'\} \xrightarrow{\{f_n\}} \{M_n,\alpha_n\}
\xrightarrow{\{g_n\}} \{M_n'',\alpha_n''\} \to 0$ be a pro-exact sequence
of
pro-$R$-modules. Then there is a natural exact sequence
\begin{multline*}
0 \to \invlim{n \ge 1}{M_n'} \xrightarrow{\invlim{n \ge 1}{f_n}}
\invlim{n \ge 1}{M_n} \xrightarrow{\invlim{n \ge 1}{g_n}}
\invlim{n \ge 1}{M_n''} \xrightarrow{\delta}
\\
 \higherlim{n \ge 1}{1}{M_n'} \xrightarrow{\higherlim{n \ge 1}{1}{f_n}}
\higherlim{n \ge 1}{1}{M_n} \xrightarrow{\higherlim{n \ge 1}{1}{g_n}}
\higherlim{n \ge 1}{1}{M_n''} \to 0.
\end{multline*}
In particular a pro-isomorphism
$\{f_n\} \colon \{M_n,\alpha_n\} \to \{N_n,\beta_n\}$ induces isomorphisms
\[
\begin{array}{llcl}
\invlim{n \ge 1}{f_n} \colon & \invlim{n \ge 1}{M_n}
& \xrightarrow{\cong} & \invlim{n \ge 1}{N_n};
\\
\higherlim{n \ge 1}{1}{f_n} \colon & \higherlim{n \ge 1}{1}{M_n}
& \xrightarrow{\cong} & \higherlim{n \ge 1}{1}{N_n}.
\end{array}
\]
\end{lemma}
\begin{proof}
If $0 \to \{M_n',\alpha_n'\} \xrightarrow{\{f_n\}} \{M_n,\alpha_n\}
\xrightarrow{g_n} \{M_n'',\alpha_n''\} \to 0$ is exact, the
construction of the six-term sequence is obvious
(see for instance~\cite[Proposition~7.63 on page~127]{Switzer(1975)}).
Hence it remains to show for a pro-trivial pro-$R$-module
$\{M_n,\alpha_n\}$ that $\invlim{n \ge 1}{M_n}$ and
$\higherlim{n \ge 1}{1}{M_n}$ vanish. This follows directly from the
standard construction for these limits as the kernel and cokernel of
the homomorphism
\[
\prod_{n \ge 1} M_n \to \prod_{n \ge 1} M_n, \quad
(x_n)_{n \ge 1} \mapsto
(x_n - \alpha_{n+1}(x_{n+1}))_{n \ge 1}.
\]
\end{proof}

\begin{lemma} \label{lem:pro-exactness_and_colimits_hom_and_ext}
Let $0 \to \{M_n',\alpha_n'\} \xrightarrow{\{f_n\}} \{M_n,\alpha_n\}
\xrightarrow{\{g_n\}} \{M_n'',\alpha_n''\} \to 0$ be a pro-exact sequence
of pro-$R$-modules. Then there is a natural exact sequence
\begin{multline*}
0
\to \colim{n \ge 1}{\hom_{\IZ}(M_n'',\IZ)}
\to \colim{n \ge 1}{\hom_{\IZ}(M_n,\IZ)}
\\
\to \colim{n \ge 1}{\hom_{\IZ}(M_n',\IZ)}
\to \colim{n \ge 1}{\ext_{\IZ}^1(M_n'',\IZ)}
\\
\to \colim{n \ge 1}{\ext_{\IZ}^1(M_n,\IZ)}
\to \colim{n \ge 1}{\ext_{\IZ}^1(M_n',\IZ)}
\to 0.
\end{multline*}
In particular a pro-isomorphism
$\{f_n\} \colon \{M_n,\alpha_n\} \to \{N_n,\beta_n\}$ induces isomorphisms
\begin{eqnarray*}
\colim{n \ge 1}{\hom_{\IZ}(N_n,\IZ)}
& \xrightarrow{\cong} &
\colim{n \ge 1}{\hom_{\IZ}(M_n,\IZ)};
\\
\colim{n \ge 1}{\ext_{\IZ}^1(N_n,\IZ)}
& \xrightarrow{\cong} &
 \colim{n \ge 1}{\ext_{\IZ}^1(M_n,\IZ)}.
\end{eqnarray*}
\end{lemma}
\begin{proof} The proof is analogous to the one of the previous
Lemma~\ref{lem:pro-exactness_and_limits} using the fact that
$\colim{n \ge 1}{}$ is an exact functor.
\end{proof}


\typeout{-----------------------  Section 2  ------------------------}

\section{Completion and cocompletion theorems}
\label{sec:Completion_and_cocompletion_theorems}

Let $G$ be a discrete group. Denote by $K^*_G$ equivariant topological
$K$-theory. This is a multiplicative
$G$-cohomology theory for proper $G$-$CW$-complexes which comes with
various extra structures such as induction, restriction and inflation. It is
defined in terms of classifying spaces for $G$-vector bundles over
proper $G$-$CW$-complexes (see~\cite[Theorem~2.7 and Section~3]{Lueck-Oliver(2001b)}).

Recall that a $G$-$CW$-complex $X$ is proper if and only if its
isotropy groups are finite~\cite[Theorem~1.23 on
page~18]{Lueck(1989)} and is finite if and only if $X$ is cocompact,
i.e. $G\backslash X$ is compact. If one considers only finite proper
$G$-$CW$-complexes, then $K^*_G(X)$ has other descriptions which are
all equivalent. There is a construction due to
Phillips~\cite{Phillips(1989)} in terms of infinite dimensional
vector bundles. In L\"uck and
Oliver~\cite[Theorem~3.2]{Lueck-Oliver(2001a)} it is shown that it
suffices to use finite dimensional $G$-vector bundles and that one
can give a definition in terms of the Grothendieck group $\IK_G(X)$
of the monoid of isomorphism classes of $G$-vector bundles over $X$.
In case where $G$ is the trivial group, we just write $\IK(X)$. One
can also define for a finite proper $G$-$CW$-complex $X$ the
equivariant topological $K$-theory as the topological $K$-theory
$K_*(C_0(X) \rtimes G)$ of the crossed product $C^*$-algebra $C_0(X)
\rtimes G$ (see~\cite[Theorem~6.7 on page~96]{Phillips(1989)}) ,
while equivariant topological $K$-homology of $X$ (by the dual of
the Green-Julg theorem~\cite[Theorem~20.2.7~(b)]{Blackadar(1998)})
also can be defined as the equivariant $KK$-group
$KK^G_*(C_0(X),\IC)$. Here $C_0(X)$ denotes the $C^*$-algebra of
continuous function on $X$ vanishing at infinity.

For any proper $G$-$CW$-complex $X$ the is a natural ring homomorphism
\[
\gamma_G(X) \colon \IK_G(X) \to K_G^0(X)
\]
which allows to regard $K_G^p(X)$ as a $\IK_G(X)$-module in the sequel.

\begin{definition}[Augmentation ideal]
\label{def:augmentation_ideal}
The augmentation ideal $\AI_G(Y)\subseteq \IK_G(Y)$ is given by the set of
elements in $\IK_G(Y)$ represented by virtual $G$-vector bundles of
dimension zero on all components of $Y$.

If $G$  is trivial we just write $\AI(Y)$.
\end{definition}

 We have the following easy
but crucial lemma (cf.~Lemma 4.2 in~\cite{Lueck-Oliver(2001a)}).

\begin{lemma} \label{lem:In_is_0_for_dim_le_n-1}
Let $Z$ be a $CW$-complex of dimension $n-1$.
Then the $n$-fold product of elements in $\AI(Z)\subseteq \IK(Z)$ is zero.
\end{lemma}

Now fix a finite proper $G$-$CW$-complex $X$, and a map $f\colon X\to L$
to a finite dimensional proper $G$-$CW$-complex $L$ whose isotropy
subgroups have bounded order. We obtain a ring homomorphism
$f^* \colon \IK_G(L) \to \IK_G(X)$ by the pullback construction.
Hence we can regard $K^*_G(X)$ as a module over the ring
$\IK_G(L)$. Put $I = \AI_G(L)$. For any natural number
$n\ge 1$,   consider the composite
\begin{multline*}
I^n\cdot K^*_G(X) \subseteq K^*_G(X)
\xrightarrow{\pr^*} K_G^*(EG \times X)
\xrightarrow{\cong} K^*(EG \times_G X)\\
\xrightarrow{K^*(i_{n-1})} K^*((EG \times_G X)^{n-1}),
\end{multline*}
where $\pr\colon  EG \times X \to X$ is the projection,
$i_{n-1} \colon (EG \times_G X)^{n-1} \to EG
\times_G X$
is the inclusion of the $(n-1)$-skeleton  and the
isomorphism $K_G^*(EG \times X)
\xrightarrow{\cong} K^*(EG \times_G X)$ comes from dividing
out the (free proper) $G$-action~\cite[Proposition~3.3]{Lueck-Oliver(2001b)}.
This composite is trivial, since the image is contained in the ideal
which is generated by the set $\AI((EG \times_G X)^{n-1})^n$,
and the latter is trivial, since
$\AI((EG \times_G X)^{n-1})^n=0$
 by Lemma~\ref{lem:In_is_0_for_dim_le_n-1}.
The composites therefore define a pro-homomorphism
\begin{eqnarray}
\lambda^{X,f}\colon  \{K^*_G(X)/I^n\cdot K^*_G(X)\}
& \to & \{K^*((EG \times_G X)^{n-1})\},
\label{pro-homomorphism_lambda(X,f)}
\end{eqnarray}
where the structure maps on the left side are given by the obvious
projections  and on the right side are induced by the various
inclusions of the skeletons. The following theorem is taken from
L\"uck-Oliver~\cite[Theorem6.5]{Lueck-Oliver(2001b)}),~\cite[Theorem~4.3]{Lueck-Oliver(2001a)}).
\begin{theorem}[Completion Theorem]
\label{the:Completion_Theorem}
Let $G$ be a discrete group. Let $X$ be a finite proper
$G$-$CW$-complex and let $L$ be a finite dimensional
proper $G$-$CW$-complex whose isotropy subgroups have bounded order.
Fix a $G$-map $f\colon X\to L$ and regard $K^*_G(X)$ as a module over
$\IK_G(L)$. Moreover, let $I=\AI_G(L)$ be the augmentation ideal.

Then
\[
\lambda^{X,f}\colon \{K^*_G(X)/I^n\cdot K^*_G(X)\} \to
 \{K^*((EG \times_G X)^{n-1})\}
\]
is a pro-isomorphism of pro-$\IZ$-modules.

The inverse system $\{K^*_G(X)/I^n\cdot K^*_G(X)\}$ satisfies the
Mittag-Leffler condition. In particular
\[
\higherlim{n \ge 1}{1}{K^{*-1 }(EG \times_G X)^n)}  =  0,
\]
and $\lambda^{X,f}$ and the various inclusions
$i_n\colon (EG \times_G X)^n \to EG \times_G X$
induce isomorphisms
\[
K^*_{G}(X)\widehat{_I} \xrightarrow{\cong} K^*(EG \times_{G} X)
\xrightarrow{\cong} \invlim{n \ge 1}{K^*((EG \times_{G} X)^n)},
\]
where $K^*_{G}(X)\widehat{_I} =
\invlim{n \ge 1}{K^*_{G}(X)/I^n\cdot K^*_{G}(X)}$
is the $I$-adic completion of $K^*_{G}(X)$.
\end{theorem}

\begin{remark} \label{rem:Completion_Theorem_for_finite_Gamma_and_L_is_point} \em
In the case where $G$ is finite and $L$ is a point,  Theorem~\ref{the:Completion_Theorem} above
coincides with the Atiyah-Segal completion theorem for finite groups
(see~\cite[Theorem~2.1]{Atiyah-Segal(1969)}). The
classical Atiyah-Segal completion theorem is stated for compact Lie groups.
However, the theorem above does not hold if $G$ is replaced by a
Lie group of positive dimension (see~\cite[Section~5]{Lueck-Oliver(2001a)}).
\em
\end{remark}

We now pass to $K$-homology.
\begin{lemma} \label{lem:EGamma_times_(Gamma)_X_of_finite_type}
If $X$ is a finite proper $G$-$CW$-complex, then
$EG \times_{G} X$ is homotopy equivalent to a $CW$-complex of finite type.
\end{lemma}
\begin{proof}
We use induction over the dimension and
subinduction over the number of equivariant cells
of top dimension in $X$. The induction beginning $X = \emptyset$ is trivial.
In the induction step
we write $X$ as a pushout of a diagram
$G/H \times D^n \xleftarrow{i} G/H \times S^{n-1} \rightarrow Y$, where
$i$ is the inclusion, $Y$ a $G$-$CW$-subcomplex of $X$ and $n = \dim(X)$.
We obtain a pushout of $CW$-complexes
\[
\xycomsquareminus{EG\times_{G}
G/H \times S^{n-1}}{}{EG\times_{G} Y}
{\id_{EG} \times_{G} i}{}
{EG\times_{G}  G/H \times D^n}{}{EG\times_{G} X}
\]
with $\id_{EG} \times_{G} i$ a cofibration. Hence
$EG\times_{G} X$ has the
 homotopy type of a $CW$-complex of finite type if
$EG\times_{G}  G/H \times S^{n-1}$,
$EG\times_{G} Y$ and
$EG\times_{G}  G/H \times D^n$ have this
property. This is true for the first two
by the induction hypothesis and for the third one
since it is homotopy equivalent to $BH$.
\end{proof}

Now we can give the proof of the Cocompletion Theorem~\ref{the:Cocompletion_Theorem}

\begin{proof}[Proof of Theorem~\ref{the:Cocompletion_Theorem}]
Because of Lemma~\ref{lem:EGamma_times_(Gamma)_X_of_finite_type} we can choose
a $CW$-complex $Y$ of finite type and a cellular homotopy equivalence
$f\colon Y \to EG\times_{G} X$.
Let $f^n\colon Y^n \to (EG\times_{G} X)^n$
be the map induced on the $n$-skeletons. Notice that $f^n$ is not
necessarily a homotopy equivalence and $K^*(f^n)$ is not necessarily an isomorphism.
Nevertheless, one easily checks that we obtain a pro-isomorphism of
pro-$\IZ$-modules
\[
\{K^*(f^n)\} \colon \{K^*((EG \times_{G} X)^n)\} \to \{K^*(Y^n)\}.
\]
Thus we obtain from the Completion Theorem~\ref{the:Completion_Theorem}
a pro-isomorphism of pro-$\IZ$-modules
\[
\{K^*(f^n)\} \circ \lambda^{X,f} \colon \{K^*_{G}(X)/I^n\cdot K^*_{G}(X)\}
\to \{K^*(Y^n)\}.
\]
>From the $K$-homology version of the universal coefficient theorem for
topological $K$-theory~\ref{the:Universal_coefficient_theorem_for_K-theory}
for finite $CW$-complexes  and the fact that
$\colim{n \ge 1}$ is an exact functor, we get the exact sequence
\begin{multline*}
0 \to \colim{n \ge 1 }{\ext^1_{\IZ}(K^{*+1}(Y^{n}),\IZ )} \to K_*(Y)
\\
\to
\colim{n \ge 1}{\hom_{\IZ}(K^*(Y^{n}),\IZ)} \to 0.
\end{multline*}
The map $f$ and the pro-isomorphism $\{K^*(f^n)\} \circ \lambda^{X,f}$ induce
isomorphisms (see Lemma~\ref{lem:pro-exactness_and_colimits_hom_and_ext})
\begin{eqnarray*}
K_*(f) \colon K_*(Y) & \xrightarrow{\cong} &K_*(EG \times_{G} X);
\\
\colim{n \ge 1 }{\ext^1_{\IZ}(K^*_{G}(X)/I^n\cdot K^*_{G}(X),\IZ )}
& \xrightarrow{\cong} &
\colim{n \ge 1 }{\ext^1_{\IZ}(K^{*+1}(Y^{n}),\IZ )};
\\
\colim{n \ge 1}{\hom_{\IZ}(K^*_{G}(X)/I^n\cdot K^*_{G}(X),\IZ)}
& \xrightarrow{\cong} &
\colim{n \ge 1}{\hom_{\IZ}(K^*(Y^{n}),\IZ)}.
\end{eqnarray*}
 Combining these isomorphisms with the exact sequence above proves the Cocompletion
Theorem~\ref{the:Cocompletion_Theorem}.
\end{proof}

\begin{remark} \label{rem:continuous_version_of_the_Cocompletion_Theorem}\em The
  Cocompletion Theorem~\ref{the:Cocompletion_Theorem} can be formulated elegantly within the category of abelian topological groups and continuous homomorphisms. 
  If we equip the completion $K^*_G(X)\widehat{_I}$
  with the $I$-adic topology and $\IZ$ with the discrete topology then the set
  of continuous homomorphisms $\hom_{\cts}(K^*_G(X)\widehat{_I},\IZ)$ is
  isomorphic to $\colim{n \ge 1}{\hom_{\IZ}(K^*_{G}(X)/I^n\cdot
    K^*_{G}(X),\IZ)}$. On the other hand, although the category of topological
  abelian groups is not exact one can introduce a notion of exact sequences (in
  the sense of \cite[Section 1.1]{Yoneda(1960)}) and correspondingly a notion of
  a group of isomorphisms classes of extensions (see \cite[Corollary on page
  537]{Yoneda(1960)}). In the case at hand we get that the group of isomorphisms
  classes of extensions $\ext_{\cts}(K^*_G(X)\widehat{_I},\IZ)$, the continuous ext-group in the sense of~\cite{Yoneda(1960)}, is isomorphic to
  $\colim{n \ge 1 }{\ext^1_{\IZ}(K^*_{G}(X)/I^n\cdot K^*_{G}(X),\IZ )}$. With
  these identifications the exact sequence of the Cocompletion
  Theorem~\ref{the:Cocompletion_Theorem} reads
  \[
  0\to {\ext^1_{\cts}(K^{*+1}_G(X)\widehat{_I}, \IZ )} \to K_*(EG\times_G X) \to
  {\hom_{\cts}(K^*_G(X)\widehat{_I}, \IZ)} \to 0.
  \]
\end{remark}

\begin{remark} \label{rem:Cocompletion_Theorem_and_the_approach_of_Greenlees}\em
The Cocompletion Theorem~\ref{the:Cocompletion_Theorem} is closely
related to the local cohomology approach to equivariant $K$-homology
due to Greenlees~\cite{Greenlees(1993a)}. If $G$ is a finite group
it follows from~\cite[(4.2) and~(5.1)]{Greenlees(1993a)} that there
is a short exact sequence
\begin{equation}
\label{local.cohomology}
 0 \to H_I^1(K_{\fulldot}^G(X))_{*+1} \to K_{*}(EG \times_G X) \to H^0_I(K_{\fulldot}^G(X))_*\to 0,\
\end{equation}
where $H_I^k(M_\fulldot)$ denotes the $k$-th local cohomology of the
graded $R(G)$-module $M_\fulldot$ with respect to the augmentation
ideal $I\subset R(G)$. A precise definition of the local cohomology
groups occuring in (\ref{local.cohomology}) can be found in
\cite[Section~2]{Greenlees(1993a)}. By a Theorem of Grothendieck in
\cite{Grothendieck(1957)} (quoted in \cite{Greenlees(1993a)} as
Theorem 2.5 (ii)) one has
\[H^n_I (K_{\fulldot}^G(X))_* \cong \colim{n \ge 1}{\ext^n_{R(G)}(R(G)/I^n,K_{*}^G(X))}.
\]
Using exact sequence of the Equivariant Universal Coefficient
Theorem for $K$-homology as stated in
Remark~\ref{rem:Boekstedt_universal_coefficient_theorem}, the
adjunction \[\hom_{R(G)}(M, \hom_{R(G)}(C,N)) \cong \hom_{R(G)}(M
\otimes_{R(G)} C,N))\] for $R(G)$-modules $M,C,N$ with $C$ being
finitely generated, and the $R(G)$-module isomorphism
$\ext^{i}_{R(G)}(M,R(G)) \cong \ext^{i}_{\IZ}(M,\IZ)$ (emphasized in
Remark~\ref{rem:Boekstedt_universal_coefficient_theorem}) one can
see that the exact sequence of the Cocompletion Theorem~\ref{the:Cocompletion_Theorem}
is exact if and only if~\eqref{local.cohomology} is. In particular the Cocompletion Theorem
yields an alternative proof for the exactness of~\eqref{local.cohomology}.
\end{remark}


\typeout{-----------------------  Section 3  ------------------------}

\section{Borel cohomology}
\label{sec:Borel_cohomology}

Let $\calh^*$ be  a (generalized) cohomology theory with
values in the category of
$\IZ$-modules which satisfies the \emph{disjoint union axiom}
for arbitrary index sets, i.e., for any family $\{X_i \mid i \in I\}$  the map
\[
\prod_{i \in I} \calh^k(j_i)\colon \calh^k(\coprod_{i \in I} X_i)
\xrightarrow{\cong} \prod_{i \in I} \calh^k(X_i)
\]
is an isomorphism, where $j_i \colon X_i \to \coprod_{i \in I} X_i$
is the canonical inclusion.
Any such theory $\calh^*$ is given by an $\Omega$-spectrum $\bfE$ and, vice versa,
any cohomology theory given by an $\Omega$-spectrum satisfies the disjoint union axiom.
Given a $CW$-complex $X$, let $\widetilde{\calh}^k(X)$ be the cokernel of the map
$\calh^k(\pt) \to \calh^k(X)$ induced by the projection $X \to \pt$.
Our main example for $\calh^*$ will be topological $K$-theory $K^*$.
If $M$ is an abelian group, we define the cohomology theory $\calh^*(-;M)$ by
the $\Omega$-spectrum which is the fibrant replacement of the
smash product of the spectrum associated with
$\calh^*$ with the Moore spectrum associated to $M$.
If $M$ is a ring $R$, then $\calh^*(-;R)$ takes values in the category of $R$-modules.

\begin{lemma} \label{lem:Hast(X;R)_is_0_implies_calhast(X)_is_0}
Let $X$ be a $CW$-complex such that its reduced singular cohomology
$\widetilde{H}^k(X;\calh^l(\pt))$ with
coefficients in the abelian group $\calh^l(\pt)$
vanishes for all $k \ge 0$ and $l\in \IZ$. Then

\begin{enumerate}

\item \label{lem:Hast(X;R)_is_0_implies_calhast(X)_is_0:restriction_from_Xnto_X(n-1)}
The inclusion $X^{n-1} \to X^n$ of the
$(n-1)$-skeleton into the $n$-skeleton induces the zero-map
$\widetilde{\calh}^k(X^n) \to \widetilde{\calh}^k(X^{n-1})$
for all $k \in \IZ$ and $n \ge 2$.
The pro-$\IZ$-module $\{\widetilde{\calh}^k(X^n)\}$ is pro-trivial;

\item \label{lem:Hast(X;R)_is_0_implies_calhast(X)_is_0:calhk(X)_is_0}
We have $\widetilde{\calh}^k(X) = 0$ for all $k \in \IZ$.

\end{enumerate}
\end{lemma}
\begin{proof}~\ref{lem:Hast(X;R)_is_0_implies_calhast(X)_is_0:restriction_from_Xnto_X(n-1)}
Since $X^n$ is finite dimensional, the reduced Atiyah-Hirzebruch
spectral cohomology sequence converges
to $\widetilde{\calh}^{k+l}(X^n)$. It has as
$E_2$-term $E_2^{k,l}(X^n) = \widetilde{H}^k(X^n;\calh^l(\pt))$. Since
$\widetilde{H}^k(X;\calh^l(\pt)) = 0$ for all $k$, we have $E_2^{k,l}(X^n) = 0$ for $k \not= n$.
This implies $E_{\infty}^{k,l}(X^n) = 0$ for $k \not= n$.
We have the descending filtration $F^{k,m-k} \calh^m(X^n)$ of $\calh^m(X^n)$ such that
\[
F^{k,l} \calh^m(X^n)/F^{k+1,l-1} \calh^m(X^n) \cong E_{\infty}^{k,l}(X^n).
\]
Hence $F^{k,l} \calh^m(X^n) = 0$ for $k \ge n$ and $F^{k,l} \calh^m(X^n) = \calh^m(X^n)$ for $k < n$.
Since the map $\calh^m(X^n) \to \calh^m(X^{n-1})$ respects this filtration,
it must be trivial.
\\[1mm]~\ref{lem:Hast(X;R)_is_0_implies_calhast(X)_is_0:calhk(X)_is_0}
Recall Milnor's exact sequence~\cite[Theorem 1.3 in XIII.1 on page 605]{Whitehead(1978)}
\[
0 \to \higherlim{n \to \infty}{1} {\widetilde{\calh}^{k-1}(X^n)} \to
\widetilde{\calh}^k(X) \to \invlim{n \to \infty}{\widetilde{\calh}^k(X^n)} \to 0.
\]
Since  $\{\widetilde{\calh}^k(X^n)\}$ is pro-trivial for all $k \in \IZ$, we
conclude from Lemma~\ref{lem:pro-exactness_and_limits}
\begin{eqnarray*}
\higherlim{n \to \infty}{1} {\widetilde{\calh}^{k-1}(X^n)} & = & 0;
\\
\invlim{n \to \infty}{\widetilde{\calh}^k(X^n)} & = & 0.
\end{eqnarray*}
This finishes the proof of Lemma~\ref{lem:Hast(X;R)_is_0_implies_calhast(X)_is_0}.
\end{proof}

\begin{lemma} \label{lem:commuting_calh_and_R}
Let $Y$ be a finite $CW$-complex
Let $R$ be a commutative associative ring
which is flat over $\IZ$.
Then the canonical $R$-map
\[
\calh^k(Y) \otimes_{\IZ} R \xrightarrow{\cong} \calh^k(Y;R)
\]
is an isomorphism.
\end{lemma}
\begin{proof}
Since $R$ is flat over $\IZ$,
the map  $\calh^k(Y) \otimes_{\IZ} R \xrightarrow{\cong} \calh^k(Y;R)$
is a transformation of homology theories. It is bijective for
$Y = \pt$. Hence by a Mayer-Vietoris argument it is bijective for every finite $CW$-complex $Y$.
\end{proof}

\begin{lemma} \label{lem:splitting_calhast(q)}
Let $X$ be a finite  proper $G$-$CW$-complex.
Let $\calp(X)$ be the set of primes $p$ which divide the order of some
isotropy group of $X$.  Let $\IZ \subseteq \IZ\left[\frac{1}{\calp(X)}\right]\subseteq \IQ$
be the ring obtained from $\IZ$ by inverting the elements in $\calp(X)$.
Let $q(X)\colon EG \times_G X \to G\backslash X$ be the projection. Then there is for $k \in \IZ$
a $R$-map, natural in $X$,
\[
r^k(X) \colon \calh^k(EG \times_G X)\otimes_{\IZ} \IZ\left[\frac{1}{\calp(X)}\right]
 \to \calh^k(G\backslash X)\otimes_{\IZ} \IZ\left[\frac{1}{\calp(X)}\right]
\]
such that $r^k(X)\circ \calh^k(q(X))\otimes_{\IZ} \id \colon
\calh^k(G\backslash X) ) \otimes_{\IZ}  \IZ\left[\frac{1}{\calp(X)}\right]\to
\calh^k(G\backslash X)) \otimes_{\IZ}  \IZ\left[\frac{1}{\calp(X)}\right] $
is an isomorphism.
\end{lemma}
\begin{proof} Put $\Lambda = \IZ\left[\frac{1}{\calp(X)}\right]$.
Consider the following commutative diagram
\[
\xymatrix@!C=12em{
\calh^k(G\backslash X)\otimes_{\IZ} \Lambda
\ar[r]^-{\calh^k(q(X)) \otimes_{\IZ} \Lambda}
\ar[d]
&
\calh^k(EG \times_G X)\otimes_{\IZ} \Lambda
\ar[d]
\\
\calh^k(G\backslash X;\Lambda)
\ar[r]^-{\calh^k(q(X);\Lambda)}
&
\calh^k(EG \times_G X;\Lambda)
}
\]
The left vertical arrow is an isomorphism by Lemma~\ref{lem:commuting_calh_and_R}.
Hence it suffices to show that the lower horizontal map is an isomorphism.
Since $X$ is finite proper, a Mayer-Vietoris argument shows that
it suffices to treat the case $X = G/H$ for some finite group $H \subset G$
such that $|H|$ is invertible
in $\Lambda$. Since $\widetilde{H}^k(BH;\calh^q(\pt;\Lambda)) = 0$
vanishes for all $k$ by~\cite[Corollary 10.2 in Chapter III on page 84]{Brown(1982)}, this follows from
Lemma~\ref{lem:Hast(X;R)_is_0_implies_calhast(X)_is_0}~\ref{lem:Hast(X;R)_is_0_implies_calhast(X)_is_0:calhk(X)_is_0}.
\end{proof}
\begin{lemma} \label{lem:splitting_the_pro-module_K(BH(n-1)}
Let $H$ be a finite group. Let $\calp(H)$ be the set of primes dividing $|H|$.
The canonical map of pro-$\IZ$-modules
\[
\{\widetilde{\calh}^k(BH^{n-1})\} \xrightarrow{\cong}
\prod_{p \in \calp(H)} \{\widetilde{\calh}^k(BH^{n-1};\IZ\widehat{_p})\}
\]
is a pro-isomorphism for $k \in \IZ$. The pro-module
\[
\prod_{\substack{p \; \text{prime}\\ p \not\in \calp(H)}} \{\widetilde{\calh}^k(BH^{n-1};\IZ\widehat{_p})\}
\]
is pro-trivial.
\end{lemma}
\begin{proof}
We have the exact sequence of abelian groups
\[
0 \to \IZ \xrightarrow{i}  \prod_{p \in \calp(H)} \IZ\widehat{_p} \to \cok(i) \to 0
\]
where $i$ is the product of the canonical embeddings $\IZ \to \IZ\widehat{_p}$.
It induces a long exact  sequence
\begin{multline*}
\ldots \to \widetilde{\calh}^{k-1}(BH^{n-1};\cok(i)) \to \widetilde{\calh}^k(BH^{n-1}) \to
\widetilde{\calh}^k(BH^{n-1};\prod_{p \in \calp(H)} \IZ\widehat{_p})
\\
\to  \widetilde{\calh}^k(BH^{n-1};\cok(i)) \to \ldots
\end{multline*}
and thus an exact sequence of pro-$\IZ$-modules
\begin{multline*}
\{\widetilde{\calh}^{k-1}(BH^{n-1};\cok(i))\} \to \{\widetilde{\calh}^k(BH^{n-1})\} \to
\{\widetilde{\calh}^k(BH^{n-1};\prod_{p \in \calp(H)} \IZ\widehat{_p})\}
\\
\to  \{\widetilde{\calh}^k(BH^{n-1};\cok(i))\}
\end{multline*}
Multiplication with the order of $|H|$ induces an isomorphism
$|H| \cdot \id \colon \cok(i) \xrightarrow{\cong} \cok(i)$.
Hence
$\widetilde{H}^k(BH;\calh^l(\pt;\cok(i)))$
vanishes for all $k,l \in \IZ$ by~\cite[Corollary 10.2 in Chapter III on page 84]{Brown(1982)}.
We conclude from
Lemma~\ref{lem:Hast(X;R)_is_0_implies_calhast(X)_is_0}~\ref{lem:Hast(X;R)_is_0_implies_calhast(X)_is_0:restriction_from_Xnto_X(n-1)}
that the pro-$\IZ$-module $\{\widetilde{\calh}^k(BH^{n-1};\cok(i)\}$
is trivial. This shows that the obvious map of
pro-$\IZ$-modules
\[
\{\widetilde{\calh}^k(BH^{n-1})\} \xrightarrow{\cong}
\{\widetilde{\calh}^k(BH^{n-1};\prod_{p \in \calp(H)} \IZ\widehat{_p})\}
\]
is bijective. The canonical map
\[
\widetilde{\calh}^k(Y;\prod_{p \in \calp(H)} \IZ\widehat{_p}) \to \prod_{p \in \calp(H)}
\widetilde{\calh}^k(Y;\IZ\widehat{_p})
\]
is a natural transformation of cohomology theories satisfying the disjoint union axiom
and is an isomorphism for $Y = \pt$ since the set $\calp(H)$ is finite.
Hence it is an isomorphism for every finite-dimensional $Y$ and in particular for $Y = BH^{n-1}$. We conclude that  the obvious map of
pro-$\IZ$-modules
\[
\{\widetilde{\calh}^k(BH^{n-1})\} \xrightarrow{\cong}
\prod_{p \in \calp(H)} \{\widetilde{\calh}^k(BH^{n-1};\IZ\widehat{_p})\}
\]
is bijective.

If $p$ does not divide
$|H|$, then $H^k(BH;\widetilde{\calh}^q(\pt;\IZ\widehat{_p})) = 0$
by~\cite[Corollary 10.2 in Chapter III on page 84]{Brown(1982)}.
We conclude from
Lemma~\ref{lem:Hast(X;R)_is_0_implies_calhast(X)_is_0}~\ref{lem:Hast(X;R)_is_0_implies_calhast(X)_is_0:restriction_from_Xnto_X(n-1)}
that the map induced by the inclusion
\[
\{\widetilde{\calh}^k(BH^{n};\IZ\widehat{_p})\} \to
\{\widetilde{\calh}^k(BH^{n-1};\IZ\widehat{_p})\}
\]
is trivial for all $k \in \IZ$ and $n \ge 2$. This implies that
\[
\prod_{p \not\in \calp(H)} \{\widetilde{\calh}^k(BH^{n-1};\IZ\widehat{_p})\}
\]
is pro-trivial.
\end{proof}

\begin{lemma} \label{lem:splitting_calhast(EG_times_G_X)}
Let $X$ be a proper $G$-$CW$-complex. Let $\calp$ be a set of primes containing
$\calp(X)$. Then the canonical map
\[
\calh^k(q(X)\colon EG \times_G X \to G\backslash X) \xrightarrow{\cong}
\prod_{p \in \calp} \calh^k(q(X)\colon EG \times_G X \to G\backslash X;\IZ\widehat{_p})
\]
is an isomorphism.
\end{lemma}
\begin{proof}
We conclude from the Milnor's exact  sequence~\cite[Theorem~1.3 in~XIII.1 on page~605]{Whitehead(1978)}
and the Five-Lemma that it
suffices to treat the case, where $X$ is finite dimensional.  Using Mayer-Vietoris sequences
the claim can be reduced to the case $X = G/H$ for some finite group
$H$ such that $\calp$ contains the set $\calp(H)$ of primes dividing the order of $H$.
So we must show  the canonical map
\[
\widetilde{\calh}^k(BH) \xrightarrow{\cong}
\prod_{p \in \calp} \widetilde{\calh}^k(BH;\IZ\widehat{_p})
\]
is bijective for any finite group $H$ with $(H) \subseteq \calp$. By Milnor's exact
sequence~\cite[Theorem~1.3 in~XIII.1 on page~605]{Whitehead(1978)} and the Five-Lemma,
it remains to show the bijectivity of
\begin{eqnarray*}
\invlim{n \to \infty}{\widetilde{\calh}^k(BH^{n-1})}
& \to &
\invlim{n \to \infty}{\prod_{p \in \calp} \widetilde{\calh}^k(BH^{n-1};\IZ \widehat{_p})};
\\
\higherlim{n \to \infty}{1}{\widetilde{\calh}^k(BH^{n-1})}
& \to &
\higherlim{n \to \infty}{1}{\prod_{p \in \calp} \widetilde{\calh}^k(BH^{n-1};\IZ \widehat{_p})}.
\end{eqnarray*}
This follows from Lemma~\ref{lem:pro-exactness_and_limits}
and Lemma~\ref{lem:splitting_the_pro-module_K(BH(n-1)}.
\end{proof}

For a map $f \colon X \to Y$ and a cohomology theory $\calh^*$
define $\calh^k(f)$ to be $\calh^k(\cyl(f),X)$, where $\cyl(f)$ is the mapping cylinder of $f$.

\begin{theorem}[Cohomology of the Borel construction]
\label{the:computing_calh_upper_ast(EG_times_G_X)_for_proper_X}
Let $\calh^*$ be a cohomology theory which satisfies the disjoint union axiom.
Let $X$ be a proper $G$-$CW$-complex. Let $\calp(X)$ be the set of primes $p$
for which $p$ divides the order $|G_x|$ of the isotropy group $G_x$ of the some
$x \in X$.

\begin{enumerate}

\item \label{the:computing_calh_upper_ast(EG_times_G_X)_for_proper_X:long_exact_sequence}
There is a natural long exact sequence
\begin{multline*}
\ldots \to \calh^k(G\backslash X) \xrightarrow{\calh^k(q(X))}
\calh^k(EG \times_G X)
\\
\to \prod_{p \in \calp(X)}
\calh^{k+1}(q(X) \colon EG \times_G X \to G\backslash X;\IZ\widehat{_p}) \xrightarrow{\delta^k}
 \calh^{k+1}(G\backslash X) \xrightarrow{\calh^{k+1}(q(X))} \ldots
\end{multline*}

\item \label{the:computing_calh_upper_ast(EG_times_G_X)_for_proper_X:splitting}
Suppose that $X$ is a finite proper $G$-$CW$-complex.

Then for all $k \in \IZ$ the map $\calh^k(q(X))$ appearing in
assertion~\ref{the:computing_calh_upper_ast(EG_times_G_X)_for_proper_X:long_exact_sequence}
becomes split injective after applying $- \otimes_{\IZ}  \IZ\left[\frac{1}{\calp(X)}\right]$
and we obtain a natural isomorphism
\begin{multline*}
\calh^k(EG \times_G X) \otimes_{\IZ}  \IZ\left[\frac{1}{\calp(X)}\right]  \xrightarrow{\cong}
\\
\calh^k(G\backslash X) \otimes_{\IZ} \IZ\left[\frac{1}{\calp(X)}\right]  \times
\prod_{p \in \calp(X)}
\calh^{k+1}(q(X) \colon EG \times_G X \to G\backslash X;\IZ\widehat{_p})
\otimes_{\IZ} \IZ\left[\frac{1}{\calp(X)}\right];
\end{multline*}

\item \label{the:computing_calh_upper_ast(EG_times_G_X)_for_proper_X:under_finiteness_conditions}
Suppose that $X$ is a finite proper $G$-$CW$-complex. Suppose that
$\calh^k(\pt)$ is  finitely generated as abelian groups
for all $k \in \IZ$. Assume that
$\widetilde{\calh}^k(BH;\IZ\widehat{_p})$ is a finitely generated $\IZ\widehat{_p}$-module
for all $k \in \IZ$ and all isotropy groups $H$ of $X$.

Then for all $k \in \IZ$ the abelian group
$\calh^k(G\backslash X)$ is finitely generated, and for appropriate natural numbers $r_p^k(X)$
there is an exact sequence
\begin{multline*}
0 \to A \to \calh^k(G\backslash X) \to \calh^k(EG \times_G X)
\to
B \times \prod_{p \in \calp(X)} \left(\IZ\widehat{_p}\right)^{r_p^k(X)} \to C \to 0,
\end{multline*}
where $A$, $B$ and $C$ are finite abelian groups with
\[
A \otimes_{\IZ} \IZ\left[\frac{1}{\calp(X)}\right] =
B \otimes_{\IZ}  \IZ\left[\frac{1}{\calp(X)}\right] =
C \otimes_{\IZ}  \IZ\left[\frac{1}{\calp(X)}\right] =0.
\]

\end{enumerate}
\end{theorem}
\begin{proof}~\ref{the:computing_calh_upper_ast(EG_times_G_X)_for_proper_X:long_exact_sequence}
This follows from
Lemma~\ref{lem:splitting_calhast(EG_times_G_X)} and the long exact sequence
associated to $q(X)\colon EG \times_G X \to G\backslash X$.
\\[1mm]~\ref{the:computing_calh_upper_ast(EG_times_G_X)_for_proper_X:splitting}
This follows from assertion~\ref{the:computing_calh_upper_ast(EG_times_G_X)_for_proper_X:long_exact_sequence} and
Lemma~\ref{lem:splitting_calhast(q)}.
\\[1mm]~\ref{the:computing_calh_upper_ast(EG_times_G_X)_for_proper_X:under_finiteness_conditions}
Since by assumption $X$ is a finite proper $G$-$CW$-complex and the $\IZ\widehat{_p}$-module
$\calh^k(q(G/H): EG \times_G G/H \to G\backslash (G/H);\IZ\widetilde{_p}) =\widetilde{\calh}^k(BH;\IZ\widehat{_p})$
is finitely generated  for all $k \in \IZ$ and all isotropy groups $H$ of $X$, the  $\IZ\widehat{_p}$-module
$\calh^k(q(X): EG \times_G X \to G\backslash X;\IZ\widetilde{_p})$
is  finitely generated for all $k \in \IZ$. Since $\calh^k(\pt)$ is a finitely generated
abelian group by assumption for all $k \in \IZ$ and $G\backslash X$ is a finite $CW$-complex,
the abelian group $\calh^k(G\backslash X)$
is finitely generated for $k \in \IZ$.

Recall that $\IZ\widehat{_p}$ is a principal
ideal domain and for each prime ideal $I$ there is $l \in \IN_0$ such that
$I$ is isomorphic to $(p^l)$  and $\IZ\widehat{_p}/I$ is isomorphic to $\IZ/p^l$.
Hence for any $k\in\IZ$ we have an isomorphism
\[ \calh^{k+1}(q(X) \colon EG \times_G X \to G \backslash X; \IZ\widehat{_p}) \cong
B_p \times (\IZ\widehat{_p})^{r^p_k(X)} \]
for some finite abelian $p$-group $B_p$ and some natural number $r_p^k(X)$.
Taking the product over the primes $p\in \calp(X)$,
 we get from Lemma~\ref{lem:splitting_calhast(EG_times_G_X)}
\[ \calh^{k+1}(q(X) \colon EG \times_G X \to G \backslash X) \cong \prod_{p\in\calp(X)}
B_p \times (\IZ\widehat{_p})^{r^k_p(X)}.\]
Since $\calp(X)$ is finite,
\[B := \prod_{p\in \calp(X)} B_p\]
is a finite abelian group which vanishes after inverting the primes in $\calp(X)$ and we have
\[ \calh^{k+1}(q(X) \colon EG \times_G X \to G \backslash X) \cong B \times \prod_{p\in\calp(X)}
 (\IZ\widehat{_p})^{r^k_p(X)}.\]
We obtain from assertion~\ref{the:computing_calh_upper_ast(EG_times_G_X)_for_proper_X:long_exact_sequence}
the long exact  sequence
\begin{multline*}
0 \to A \to \calh^k(G\backslash X) \to \calh^k(EG \times_G X)
\to
B \times \prod_{p \in \calp(X)} \left(\IZ\widehat{_p}\right)^{r_k^p(X)} \to C \to 0,
\end{multline*}
where  $A$ and $C$ can be identified with the image of boundary operators
\begin{eqnarray*}
A &\cong & \image (\delta^{k-1}\colon \calh^{k-1}(EG \times_G X \to G \backslash X) \to \calh^k(G \backslash X));\\
C & \cong & \image (\delta^k\colon \calh^{k}(EG \times_G X \to G \backslash X) \to \calh^{k+1}(G \backslash X)).
\end{eqnarray*}
We conclude from assertion~\ref{the:computing_calh_upper_ast(EG_times_G_X)_for_proper_X:splitting}
 that the image of the boundary operators vanishes after applying
$- \otimes_{\IZ} \IZ\left[\frac{1}{\calp(X)}\right]$.
Since we have already shown that the abelian group $\calh^k(G \backslash X))$ is finitely generated
for all $k \in \IZ$, $A$ and $C$ are finite abelian groups which vanish after inverting all primes in $\calp(X)$.
This finishes the proof of Theorem~\ref{the:computing_calh_upper_ast(EG_times_G_X)_for_proper_X}.
\end{proof}

\begin{theorem} \label{the:numbers_r}
Let $X$ be a finite proper $G$-$CW$-complex. If we take $\calh^*$ to be topological $K$-theory
$K^*$ in  Theorem~\ref{the:computing_calh_upper_ast(EG_times_G_X)_for_proper_X}, then the numbers
$r_p^k(X)$ appearing in assertion~\ref{the:computing_calh_upper_ast(EG_times_G_X)_for_proper_X:under_finiteness_conditions}
of Theorem~\ref{the:computing_calh_upper_ast(EG_times_G_X)_for_proper_X} are given by
\[
r_p^k(X) = \sum_{(g) \in \con_p(G)} \sum_{i \in \IZ} \dim_{\IQ}\left( H^{k+2i}(C_G\langle g \rangle\backslash X^{\langle g \rangle};\IQ)\right).
\]
\end{theorem}
\begin{proof}
Consider the equivariant cohomology theory with values in
$\IQ\widehat{_p}$ (in the sense of~\cite[Section~1]{Lueck(2007)}
which is given for a proper $G$-$CW$-complex $Y$ by
\[
\calh^k_G(Y) := K^{k+1}(q(Y) \colon EG \times Y \to G\backslash Y;\IZ\widehat{_p}) \otimes_{\IZ\widehat{_p}} \IQ\widehat{_p}.
\]
Notice that it  does not satisfy the disjoint union axiom (for arbitrary index sets) since
infinite products are not compatible with $- \otimes_{\IZ\widehat{_p}} \IQ\widehat{_p}$, but this will not matter since
we will finally consider a finite proper $G$-$CW$-complex $X$. Let $\calf(X)$ be the family
of subgroups $H$ of $G$ with $X^H \not= \emptyset$. We obtain from~\cite[Theorem~4.2]{Lueck(2007)})
(using the notation  from that paper)
an isomorphism of $\IQ\widehat{_p}$-modules
\begin{eqnarray}
\calh^k_G(X) & \xrightarrow{\cong} & \prod_{p+q = k} H^p_{\IQ\widehat{_p}\Sub(G;\calf(X))}(X;\calh^q_G(G/?)).
\label{Chern-character}
\end{eqnarray}
We get from the Atiyah-Segal Completion Theorem (see~\cite[Theorem~2.1]{Atiyah-Segal(1969)}) and~\cite[Theorem~3.5]{Lueck(2007)}
(using the notation of~\cite[Theorem~3.5]{Lueck(2007)})
pro-isomorphism of pro-$\IZ$-modules
\begin{eqnarray*}
\{\widetilde{K}^q(BH^{n-1})\} & \xrightarrow{\cong} &
\begin{cases}
\prod_{p \in \calp(H)} \{\im(\res_H^{H_p})/p^n \cdot \im(\res_H^{H_p})\} & q = 0;
\\
\{0\}  & q = 1.
\end{cases}
\end{eqnarray*}
Hence we obtain from Lemma~\ref{lem:splitting_the_pro-module_K(BH(n-1)} a
pro-isomorphism of pro-$\IZ$-modules
\begin{eqnarray*}
\prod_{p \in \calp(H)}  \{\widetilde{K}^n(BH^{n-1};\IZ\widehat{_p}))\} & \xrightarrow{\cong} &
\prod_{p \in \calp(H)} \{\im(\res_H^{H_p})/p^n \cdot \im(\res_H^{H_p})\}.
\end{eqnarray*}
One easily checks that it induces for each prime $p \in \calp(H)$ an isomorphism of
pro-isomorphism of pro-$\IZ$-modules
\begin{eqnarray*}
\{\widetilde{K}^0(BH^{n-1};\IZ\widehat{_p}))\} & \xrightarrow{\cong} &
\{\im(\res_H^{H_p})/p^n \cdot \im(\res_H^{H_p})\}.
\end{eqnarray*}
This implies that the two functors from $\Sub(G;\calf(X))$ to the category of $\IQ\widehat{_p}$-modules which send a an object $H \in \calf(X)$ to
$\calh^q(G/H) = \widetilde{K}^q(BH;\IZ\widehat{_p}) \otimes_{\IZ\widehat{_p}} \IQ\widehat{_p}$
and to $\im(\res_H^{H_p})  \otimes_{\IZ\widehat{_p}} \IQ\widehat{_p}$ respectively agree for even $q$.
For odd $q$ the functor given by $\calh^q(G/H) = \widetilde{K}^q(BH;\IZ\widehat{_p}) \otimes_{\IZ\widehat{_p}} \IQ\widehat{_p}$
is trivial.
Hence we obtain from~\eqref{Chern-character} the $\IQ\widehat{_p}$-isomorphism
\begin{eqnarray}
\calh^k_G(X) & \xrightarrow{\cong} & \prod_{i \in \IZ} H^{k+2i}_{\IQ\widehat{_p}\Sub(G)}(X;\im(\res_?^{?_p})  \otimes_{\IZ\widehat{_p}} \IQ\widehat{_p}).
\label{Chern-character_improved}
\end{eqnarray}
Now one shows analogous to the argument
in~\cite[Section~4]{Lueck(2007)} using~\cite[Theorem~5.2~(c) and Example~5.3]{Lueck(2007)} that there is an isomorphism of $\IQ\widehat{_p}$-modules
\begin{multline}
H^{k +2i}_{\IQ\widehat{_p}\Sub(G)}(X;\im(\res_?^{?_p})  \otimes_{\IZ\widehat{_p}} \IQ\widehat{_p})
\\
\cong
\prod_{(g) \in \con_p(G)} \prod_{i \in \IZ} H^{k+2i}(C_G\langle g \rangle\backslash X^{\langle g \rangle};\IQ\widehat{_p})).
\label{comp_of_bredon_coho}
\end{multline}
Now we conclude from~\eqref{Chern-character_improved} and~\eqref{comp_of_bredon_coho}.
\begin{eqnarray*}
\dim_{\IQ\widehat{_p}}(\calh_G(X))
&  =  &
\sum_{(g) \in \con_p(G)} \sum_{i \in \IZ} \dim_{\IQ\widehat{_p}}\left( H^{k+2i}(C_G\langle g \rangle\backslash X^{\langle g \rangle};\IQ\widehat{_p})\right)
\\
& = &
\sum_{(g) \in \con_p(G)} \sum_{i \in \IZ} \dim_{\IQ}\left(H^{k+2i}(C_G\langle g \rangle\backslash X^{\langle g \rangle};\IQ)\right).
\end{eqnarray*}
Since $\calh^k_G(X) = K^{k+1}(q(Y) \colon EG \times X \to G\backslash X;\IZ\widehat{_p})$ is $\IZ\widehat{_p}$-isomorphic to $(\IZ\widehat{_p})^{r^k_p(X)}$
by definition of $r^k_p(X)$, Theorem~\ref{the:numbers_r} follows.
\end{proof}


\typeout{-----------------------  Section 4  ------------------------}

\section{Borel homology}
\label{sec:Borel_homology}

The material of Section~\ref{sec:Borel_cohomology}
has analogues for Borel homology. We begin with the analogue
of Theorem~\ref{the:computing_calh_upper_ast(EG_times_G_X)_for_proper_X}.
Let $\calh_*$ be  a (generalized) homology theory with
values in the category of
$\IZ$-modules which satisfies the \emph{disjoint union axiom}
for arbitrary index sets, i.e., for any family $\{X_i \mid i \in I\}$  the map
\[
\bigoplus_{i \in I} \calh_k(j_i)\colon \bigoplus _{i \in I} \calh_k(X_i)
\xrightarrow{\cong} \calh_k (\coprod_{i \in I} X_i)
\]
is an isomorphism, where $j_i \colon X_i \to \coprod_{i \in I} X_i$
is the canonical inclusion.
Given a $CW$-complex $X$, let $\widetilde{\calh}_k(X)$
be the kernel of the map
$\calh_k(X) \to \calh_k(\pt)$ induced by the projection $X \to \pt$.
We define for any abelian group $A$
\[
\calh_k(X;A) :=  \calh_{k-d}(X \times M(A,d),X \times \pt)
\]
where $d$ is some positive integer and
 $M(A,d)$ is the Moore space associated to $A$ in degree $d$.

\begin{theorem} [Homology of the Borel construction] \label{the:computing_calh_lower_ast(EG_times_G_X)_for_proper_X}
Let $\calh_*$ be a homology theory which satisfies the disjoint union axiom.
Let $X$ be a proper $G$-$CW$-complex.
Let $\calp(X)$ be the set of primes $p$
for which $p$ divides the order $|G_x|$ of the isotropy group $G_x$
of the some $x \in X$. Then:

\begin{enumerate}

\item \label{the:computing_calh_lower_ast(EG_times_G_X)_for_proper_X:long_exact_sequence}
There is a long exact sequence
\begin{multline*}
\ldots \to \calh_{k+1}(G\backslash X) \to
\bigoplus_{p \in \calp(X)} \calh_{k+1}(q(X) \colon EG \times_G X \to G\backslash X;\IZ/{p^\infty}) \to
\\
\to \calh_k(EG \times_G X) \xrightarrow{\calh_{k}(q(X))}
 \calh_{k}(G\backslash X) \to \ldots;
\end{multline*}

\item \label{the:computing_calh_lower_ast(EG_times_G_X)_for_proper_X:splitting}
The map $\calh_{k}(q(X))$ appearing in
assertion~\ref{the:computing_calh_lower_ast(EG_times_G_X)_for_proper_X:long_exact_sequence}
induces after applying $- \otimes_{\IZ}  \IZ\left[\frac{1}{\calp(X)}\right]$ a natural isomorphism
\begin{equation*}
\calh_k(q(X)) \otimes_{\IZ}  \id \colon
\calh_k(EG \times_G X) \otimes_{\IZ}  \IZ\left[\frac{1}{\calp(X)}\right]  \xrightarrow{\cong}
\calh_k(G\backslash X) \otimes_{\IZ} \IZ\left[\frac{1}{\calp(X)}\right];
\end{equation*}

\item \label{the:computing_calh_lower_ast(EG_times_G_X)_for_proper_X:under_finiteness_conditions}
Suppose that $X$ is a finite proper $G$-$CW$-complex. Suppose that
$\calh_k(\pt)$ is  finitely generated as abelian groups
for all $k \in \IZ$. Assume that $\widetilde{\calh}_k(BH;\IZ/{p^\infty})$ can be embedded into
$(\IZ/{p^\infty})^r$ for some $r= r(k,H)$ for all $k \in \IZ$ and all isotropy
groups $H$ of $X$.

Then, for appropriate natural numbers $r_k^p(X)$, there is an exact sequence
\begin{multline*}
0 \to C \to \bigoplus_{p \in \calp(X)} (\IZ/{p^\infty})^{r_k^p(X)} \times B
\to \\
 \to \calh_k(EG \times_G X)  \xrightarrow{\calh_{k}(q(X))}  \calh_k(G\backslash X)
\to A \to 0,
\end{multline*}
where $A$, $B$ and $C$ are finite abelian groups with
\[
A \otimes_{\IZ} \IZ\left[\frac{1}{\calp(X)}\right] =
B \otimes_{\IZ}  \IZ\left[\frac{1}{\calp(X)}\right] =
C \otimes_{\IZ}  \IZ\left[\frac{1}{\calp(X)}\right] =0.
\]
\end{enumerate}
\end{theorem}
\begin{proof}~\ref{the:computing_calh_lower_ast(EG_times_G_X)_for_proper_X:long_exact_sequence}
There is a canonical isomorphism
\[
\bigoplus_{p \in \calp(X)} \IZ/p^{\infty} \xrightarrow{\cong} \bigoplus_{p \in \calp(X)} \IZ[1/p]/\IZ
\xrightarrow{\cong} \IZ\left[\frac{1}{\calp(X)}\right]/\IZ.
\]
Thus we obtain a short exact sequence of abelian groups
\[
1 \to \IZ \to \IZ\left[\frac{1}{\calp(X)}\right] \to \bigoplus_{p \in \calp(X)} \IZ/p^{\infty} \to 1.
\]
The boundary of the associated Bockstein sequence
yields a  natural transformation of
equivariant homology theories for proper $G$-$CW$-complexes satisfying the disjoint union axiom
\begin{multline*}
\partial_k(X) \colon \bigoplus_{p \in \calp(X)}
\calh_k(q(X) \colon EG \times_G X \to G\backslash X;\IZ/p^{\infty})
\\
\to \calh_{k-1}(q(X) \colon EG \times_G X \to G\backslash X).
\end{multline*}
If $H$ is a finite subgroup and $\calp$ a set of primes
which contains all primes dividing the order of $H$, then
\[
\widetilde{H}_k(BH,M) \otimes_{\IZ} \IZ\left[\frac{1}{\calp}\right] =
\widetilde{H}_k\left(BH,M \otimes_{\IZ} \IZ\left[\frac{1}{\calp}\right]\right)  =  0
\]
holds for all $k \in \IZ$ and all abelian groups $M$~\cite[Corollary 10.2 in Chapter III on page 84]{Brown(1982)}.
By the Atiyah-Hirzebruch spectral sequence we conclude that
\[\widetilde{\calh}_{k-1}\left(BH;\IZ\left[\frac{1}{\calp(X)}\right]\right))
\cong \calh_{k}\left(EG \times_G G/H \to G\backslash (G/H);\IZ\left[\frac{1}{\calp(X)}\right]\right)
\]
vanishes for all $k  \in \IZ$ and all subgroups $H \subseteq$ that appear as  isotropy group of $X$.
This implies that $\calh_k\left(q(X) \colon EG \times_G X \to G\backslash X;\IZ\left[\frac{1}{\calp(X)}\right]\right)$ vanishes for all $k \in \IZ$.
Hence  $\partial_k(X)$ is an isomorphism.
\\[1mm]~\ref{the:computing_calh_lower_ast(EG_times_G_X)_for_proper_X:splitting}
Notice that the functor $- \otimes_{\IZ} \IZ\left[\frac{1}{\calp(X)}\right]$ is exact.
We have already shown that
\[
\calh_k(q(X) \colon EG \times_G X \to G\backslash X) \otimes_{\IZ} \IZ\left[\frac{1}{\calp(X)}\right]
\cong \calh_k\left(EG \times_G X \to G\backslash X;\IZ\left[\frac{1}{\calp(X)}\right]\right)
\]
vanishes for all $k \in \IZ$. Now the claim follow from the long exact homology sequence associated to
$q(X \colon EG \times_G X \to G\backslash X)$.
\\[1mm]~\ref{the:computing_calh_lower_ast(EG_times_G_X)_for_proper_X:under_finiteness_conditions}~%
The proof is analogous to the proof of
Theorem~\ref{the:computing_calh_lower_ast(EG_times_G_X)_for_proper_X}~%
\ref{the:computing_calh_lower_ast(EG_times_G_X)_for_proper_X:under_finiteness_conditions},
as long one has Lemma~\ref{lem:abelian_p-groups_embedding_into_(Z/pinfty)r}
available which we explain next.

Fix a prime $p$. For an abelian group $A$ define $A^* := \hom_{\IZ}(A,\IQ/\IZ)$.
If $A$ is a finite abelian $p$-group, then $A$ and $A^*$ are abstractly isomorphic as
abelian groups. If $A$ is an abelian $p$-group, then $A$ is the colimit
$\colim{n \to \infty}{\ker(p^n \cdot \id_A \colon A \to A)}$ and hence
$A^* = \invlim{n\to \infty} (\ker(p^n \cdot \id_A)^*$. Hence $A^*$ is the inverse limit of a system of
finite abelian $p$-groups and therefore carries a canonical $\IZ\widehat{_p}$-module structure and
a canonical structure of a totally disconnected compact topological Hausdorff group.
\end{proof}

\begin{lemma} \label{lem:abelian_p-groups_embedding_into_(Z/pinfty)r}\
\begin{enumerate}
\item \label{lem:abelian_p-groups_embedding_into_(Z/pinfty)r:characterization}
The following assertions are equivalent for an abelian $p$-group $A$:
\begin{enumerate}
\item $A$ can be embedded into $(\IZ/{p^\infty})^r$ for some $r=r(A)$;

\item $A\cong (\IZ/{p^\infty})^r\times T$  for some $r=r(A)$ and some finite abelian
  $p$-group $T$;

\item $A^*$ is a finitely generated $\IZ\widehat{_p}$-module;

\end{enumerate}

\item \label{lem:abelian_p-groups_embedding_into_(Z/pinfty)r:_Serre_category}
The category of  groups $A$, which embed into $(\IZ/{p^\infty})^r$ for some $r=r(A)$,
form a Serre category, i.e., given an exact sequence of abelian groups
$0 \to A \to B \to C \to 0$, then $B$ belongs to the category if and only if both $A$ and $C$ do.

\end{enumerate}

\end{lemma}
\begin{proof}
Since the $\IZ$-module $\IQ/\IZ$  is divisible and hence injective, an exact sequence
$0 \to A \xrightarrow{i} B \xrightarrow{p}  C \to 0$ of abelian $p$-groups induces
an exact sequence of totally disconnected compact topological groups
$0 \to C^* \xrightarrow{p^*} B^*  \xrightarrow{i^*} A^* \to 0$.

For a topological group $G$ let $\widehat{G}$ be its Pontryagin dual, i.e., the abelian group of
continuous homomorphisms from $G$ to $S^1$. Given an exact sequence
$0 \to G \xrightarrow{j} H \xrightarrow{q}  K \to 0$ of
totally disconnected compact topological Hausdorff groups, we obtain an induced
exact sequence $0 \to \widehat{K} \xrightarrow{\widehat{q}} \widehat{H} \xrightarrow{\widehat{j}}
\widehat{G}$. For an abelian $p$-group $A$ we obtain a canonical homomorphism
$\phi_A \colon A \to \widehat{A^*}$ sending $a$ to the continuous group homomorphism
$\hom_{\IZ}(A,\IQ/\IZ) \to S^1, \hspace{2mm} f \mapsto \exp(2\pi i f(a))$. Thus we obtain for
any short exact sequence $0 \to A \xrightarrow{i} B \xrightarrow{p}  C \to 0$ of abelian $p$-groups
and commutative diagram with exact rows
\[
\xymatrix{
0 \ar[r]
&
A \ar[d]^{\phi_A} \ar[r]^{i}
&
B \ar[d]^{\phi_B} \ar[r]^{p}
&
C \ar[d]^{\phi_C} \ar[r]
&
0
\\
0 \ar[r]
&
\widehat{A^*} \ar[r]^{\widehat{i^*}}
&
\widehat{B^*} \ar[r]^{\widehat{p^*}}
&
\widehat{C^*}
&
}
\]

One easily checks that for a finite abelian $p$-group $A$ the canonical map $\phi_A$ is bijective.
Next we show that for any abelian $p$-group $B$ the map $\phi_B$ is injective.
Consider any element $b \in B$. Let $A \subseteq B$ be the finite cyclic subgroup generated by $b$.
We conclude from the commutative diagram above that $\phi_B(b) = 0$ implies that
$\phi_A(b) = 0 $ and hence $b = 0$.

Now suppose that $0 \to A \xrightarrow{i} B \xrightarrow{p}  C \to 0$
is an exact sequence of abelian $p$-groups.
Recall that $\IZ\widehat{_p}$ is a principal ideal domain. Hence $B^*$ is a finitely generated
$\IZ\widehat{_p}$-module, if and only if both $A^*$ and $C^*$ are finitely generated
$\IZ\widehat{_p}$-modules. This shows that the category of abelian
$p$-groups $B$ for which $B^*$ is a finitely generated
$\IZ\widehat{_p}$-module  is a Serre category.

Let $A$ be an abelian group which can be embedded into
$B = (\IZ/p^{\infty})^r$. Let  $0 \to A \xrightarrow{i} B \xrightarrow{p}  C \to 0$
be the  exact sequence of abelian $p$-groups with
$B = (\IZ/p^{\infty})^r$ and $C = B/A$.  We want to show that $\widehat{p^*}$ is surjective.
Recall that $\IZ\widehat{_p}$ is a principal ideal domain and the prime ideals different from $\{0\}$ and
$\IZ\widehat{_p}$ look like
$(p^n)$ for $n$ running through the positive integers.
We can find isomorphisms of $\IZ\widehat{_p}$-modules $B^* \cong (\IZ\widehat{_p})^a \bigoplus (\IZ\widehat{_p})^b$ and
$C^* \cong (\IZ\widehat{_p})^a$ such that the map $p^*\colon C^* \to B^*$ looks  under these identifications like
\[
(\IZ\widehat{_p})^a \to (\IZ\widehat{_p})^a  \bigoplus (\IZ\widehat{_p})^b, \hspace{4mm}
(x_1, x_2, \ldots , x_a) \mapsto (p^{n_1} x_1, p^{n_2} x_2, \ldots , p^{n_a} x_a, 0, \ldots , 0)
\]
for appropriate positive integers $n_1$, $n_2$, \ldots $n_a$. Hence it suffices to show for each positive integers $n$
that the map $p^n \cdot \id \colon \IZ\widehat{_p} \to \IZ\widehat{_p}$ induces
epimorphism between the Pontryagin duals. This induced map on the Pontryagin duals can be identified with
$p^n \cdot \id \colon \IZ/p^{\infty} \to  \IZ/p^{\infty}$ which is indeed surjective. Hence we obtain
a commutative diagram with exact rows
\[
\xymatrix{
0 \ar[r]
&
A \ar[d]^{\phi_A} \ar[r]^{i}
&
B \ar[d]^{\phi_B} \ar[r]^{p}
&
C \ar[d]^{\phi_C} \ar[r]
&
0
\\
0 \ar[r]
&
\widehat{A^*} \ar[r]^{\widehat{i^*}}
&\widehat{B^*} \ar[r]^{\widehat{q^*}}
& \widehat{C^*} \ar[r]
& 0
}
\]

Since $\phi_B$ is bijective for $B = (\IZ/p^{\infty})^r$ and $\phi_C$ is injective, we conclude that
$\phi_A$ is bijective. Since $A^*$ is a finitely generated $\IZ\widehat{_p}$ module,
$\widehat{A^*}$ and hence $A$ is isomorphic to $(\IZ/p^{\infty})^r \times T$
for some finite abelian $p$-group $T$ and integer $r \ge 0$.

If $A$ is isomorphic to $(\IZ/p^{\infty})^r \times T$
for some finite abelian $p$-group and integer $r \ge 0$, then $A^*$ is a finitely generated
$\IZ\widehat{_p}$-module.

Suppose that $A$ is an abelian $p$-group such that $A^*$ is a finitely generated $\IZ\widehat{_p}$-module,
then $A$ embeds into $\widehat{A^*}$ which is isomorphic to
$\IZ/p^{\infty})^s \times T$ for an integer $s \ge 0$ and
a finite abelian $p$-group $T$. Hence $A$ embeds into
$\IZ/p^{\infty})^r$ for some integer $r \ge 0$. This finishes the proof of
Lemma~\ref{lem:abelian_p-groups_embedding_into_(Z/pinfty)r} and thus of
Theorem~\ref{the:computing_calh_lower_ast(EG_times_G_X)_for_proper_X}.
\end{proof}


\typeout{-----------------------  Section 5  ------------------------}

\section{Universal coefficient theorems for $K$-theory}
\label{sec:Universal_coefficient_theorems_for_K-theory}

A proof of the following Universal Coefficients Theorem can be found
in~\cite{Anderson(1964)} and~\cite[(3.1)]{Yosimura(1975)}, the
homological version then follows from~\cite[Note~9
and~15]{Adams(1969b)}.

\begin{theorem}[Universal Coefficient Theorem for $K$-theory]
\label{the:Universal_coefficient_theorem_for_K-theory}
For any $CW$-complex $X$ there is a short exact sequence
\[
0 \to \ext_{\IZ}(K_{*-1}(X),\IZ) \to K^*(X) \to \hom_\IZ(K_*(X),\IZ) \to 0.
\]
If $X$ is a finite $CW$-complex, there is also the $K$-homological version
\[
0 \to \ext_{\IZ}(K^{*+1}(X),\IZ) \to K_*(X) \to \hom_\IZ(K^*(X),\IZ) \to 0.
\]
\end{theorem}

\begin{corollary} \label{cor:UCT.for.Borel.theories}
For any $G$-$CW$-complex $X$ there is a short exact sequence
\begin{multline*}
0 \to \ext_{\IZ}(K_{*-1}(EG\times_G X),\IZ)
\to K^*(EG\times_G X)
\\
\to \hom_\IZ(K_*(EG \times_G X),\IZ) \to 0.
\end{multline*}
\end{corollary}

Also the homological version of the Universal Coefficient Theorem has an equivariant counterpart.

\begin{remark}\em
Recall that the Completion Theorem~\ref{the:Completion_Theorem} for
 a finite proper $G$-$CW$-complex $X$ yields an isomorphism
$K^*(EG \times_{G} X) \xrightarrow{\cong} \invlim{n \ge 1}{K^*((EG
\times_{G} X)^n)}$. Hence $K^*(EG \times_{G} X)$ can be regarded as
a pro-discrete group. Thus it carries a topology, the so-called
skeletal topology. In terms of topological abelian groups the main
statement of the Completion Theorem~\ref{the:Completion_Theorem}
says that there is a canonical isomorphism of topological
groups $K^*(EG \times_{G} X) \cong K^*_{G}(X)\widehat{_I}$, where
$K^*_{G}(X)\widehat{_I}$ carries the $I$-adic topology. The exact
sequence introduced in
Remark~\ref{rem:continuous_version_of_the_Cocompletion_Theorem}
then provides an Equivariant Universal Coefficient Theorem
for $K$-homology for finite proper $G$-$CW$-complexes $X$, which
says that the following sequence is exact
\[
0\to {\ext^1_{\cts}(K^{*+1}(EG\times_G X), \IZ )} \to
K_*(EG\times_G X) \to {\hom_{\cts}(K^*(EG\times_G X), \IZ)} \to 0.
\]
\end{remark}

The following lemma will be needed for the proof of Theorem~\ref{the:intro}, which we will give below.

\begin{lemma} \label{lem:comparing_r_p_upper_kand_r_p_upper_for_K-theory}
Suppose that $X$ is a finite proper $G$-$CW$-complex.

Then the two numbers $r_p^k(X)$ and $r^p_{k-1}(X)$ defined in
Theorem~\ref{the:computing_calh_upper_ast(EG_times_G_X)_for_proper_X}~%
\ref{the:computing_calh_upper_ast(EG_times_G_X)_for_proper_X:under_finiteness_conditions} and
Theorem~\ref{the:computing_calh_lower_ast(EG_times_G_X)_for_proper_X}~%
\ref{the:computing_calh_lower_ast(EG_times_G_X)_for_proper_X:under_finiteness_conditions}
 coincide for all primes $p$ and all $k\in \IZ$.
\end{lemma}
\begin{proof}
For an abelian group $A$ let $A\widehat{_p}$ be its $p$-adic completion, i.e.,
the inverse limit $\invlim{n \to \infty}{A/p^nA}$. There is a canonical
$\IZ\widehat{_p}$-module structure on $A\widehat{_p}$. Define
\[
\dim\!\widehat{_p}(A)  :=  \dim_{\IQ\widehat{_p}}
(A\widehat{_p} \otimes_{\IZ\widehat{_p}} \IQ\widehat{_p} ).
\]
One easily checks
\begin{eqnarray}
\dim\!\widehat{_p}(\IZ) & = & 1; \label{dim_p(Z)}\\
\dim\!\widehat{_p}(\IZ\widehat{_p}) & = & 1; \label{dim_p(Z_p)}\\
\dim\!\widehat{_p}(\IZ\widehat{_q}) & = & 0, \hspace{5mm} \text {if } p \not= q; \label{dim_p(Z_q)}\\
\dim\!\widehat{_p}(A) & = & 0, \hspace{5mm} \text {if } A \text { is finite}. \label{dim_p(finite)}
\end{eqnarray}
Next  want to show for an exact sequence of abelian groups
$0 \to A \xrightarrow{i} B \xrightarrow{p} C \to 0$
\begin{eqnarray}
\dim\!\widehat{_p}(B) & = & \dim\!\widehat{_p}(A) + \dim\!\widehat{_p}(C).
\label{Additivity_of_dim_p}
\end{eqnarray}
We obtain an exact sequence
of $\IZ\widehat{_p}$-modules  $0 \to A\widehat{_p} \xrightarrow{i\widehat{_p}}
B\widehat{_p} \xrightarrow{p\widehat{_p}} C\widehat{_p} \to 0$
from~\cite[Corollary 10.3]{Atiyah-McDonald(1969)}. The induced sequence of
$\IQ\widehat{_p}$-modules
$0 \to A\widehat{_p} \otimes_{\IZ\widehat{_p}} \IQ\widehat{_p}
\xrightarrow{i\widehat{_p}  \otimes_{\IZ\widehat{_p}} \id}
B\widehat{_p} \otimes_{\IZ\widehat{_p}} \IQ\widehat{_p}
 \xrightarrow{p\widehat{_p} \otimes_{\IZ\widehat{_p}} \id}
C\widehat{_p}  \otimes_{\IZ\widehat{_p}} \IQ\widehat{_p}\to 0$
is exact. Now~\eqref{Additivity_of_dim_p} follows.
We conclude from~\eqref{dim_p(Z)},~\eqref{dim_p(finite)}
and~\eqref{Additivity_of_dim_p} for any finitely generated abelian group $A$
\begin{eqnarray}
\dim\!\widehat{_p}(A) & = & \dim_{\IZ}(A),
\label{dim_p(A)_for_A_fin.gen.}
\end{eqnarray}
where $\dim_{\IZ}(A)$ is the dimension of the rational vector space $A\otimes_{\IZ}\IQ$.

Let $C_*$ be a finite dimensional chain complex of abelian groups
such that $\dim\!\widehat{_p}(C_k)$ is finite  and
$H_k(C_*)$ is finite for each $k \in \IZ$. Then we conclude from~\eqref{dim_p(finite)}
and~\eqref{Additivity_of_dim_p}
\begin{eqnarray}
\sum_{k \in \IZ} (-1)^k \cdot \dim\!\widehat{_p}(C_k)  =  0.
\label{Euler_formula_for_C_ast}
\end{eqnarray}
Theorem~\ref{the:computing_calh_upper_ast(EG_times_G_X)_for_proper_X}~%
\ref{the:computing_calh_upper_ast(EG_times_G_X)_for_proper_X:under_finiteness_conditions} implies
that there is a $2$-dimensional chain complex
\begin{multline*}
\ldots \to 0 \to K^k(G\backslash X) \to K^k(EG \times_G X) \to
\prod_{p \in \calp(X)} \left(\IZ\widehat{_p}\right)^{r_p^k(X)} \to 0 \to \ldots,
\end{multline*}
whose homology is finite, and that $K^k(G\backslash X)$ is a finitely generated abelian group.
We conclude from~\eqref{dim_p(A)_for_A_fin.gen.} and~\eqref{Euler_formula_for_C_ast} for any prime $p$
\begin{eqnarray}
\dim\!\widehat{_p}(K^k(EG \times_G X)) & = & r_p^k(X) + \dim_{\IZ}(K^k(G\backslash X)),
\label{dim_p(K_upper_k(EG_times_G_X):I}
\end{eqnarray}
where $r^k_p(X)$ is defined to be $0$ for $p \not\in \calp(X)$.

We conclude from Theorem~\ref{the:computing_calh_lower_ast(EG_times_G_X)_for_proper_X}~%
\ref{the:computing_calh_lower_ast(EG_times_G_X)_for_proper_X:under_finiteness_conditions}
that there is a $2$-dimensional chain complex with finite homology
\begin{multline*} \ldots 0 \to \ext_{\IZ}(K_k(G\backslash X),\IZ) \to
\ext_{\IZ}(K_k(EG \times_G X,\IZ ) \\
\to \ext_{\IZ}(\bigoplus_{p \in \calp(X)} (\IZ/{p^\infty})^{r_k^p(X)},\IZ)
\to 0 \to 0 \to \ldots
\end{multline*}
Since $K_k(G\backslash X),\IZ)$ is finitely generated abelian,
$\ext_{\IZ}(K_k(G\backslash X),\IZ)$ is finite. The $\IZ\widehat{_p}$-module
$\ext(\IZ/p^{\infty},\IZ) = \ext(\IZ/p^{\infty},\IZ)\widehat{_p}$ is isomorphic to $\IZ\widehat{_p}$
(see~\cite[Example~3.3.3 on page~73]{Weibel(1994)}). We conclude
from~\eqref{dim_p(finite)} and~\eqref{Euler_formula_for_C_ast} for any prime $p$
\begin{eqnarray}
\dim\!\widehat{_p}(\ext_{\IZ}(K_k(G\backslash X),\IZ)) & = & r_k^p(X).
\label{dim_p(ext(K_k,Z))}
\end{eqnarray}
 Theorem~\ref{the:computing_calh_lower_ast(EG_times_G_X)_for_proper_X}~%
\ref{the:computing_calh_lower_ast(EG_times_G_X)_for_proper_X:under_finiteness_conditions}
implies that the map
\[
\hom_{\IZ}(K_k(q(X)),\IZ) \colon \hom_{\IZ}(K_k(G\backslash X),\IZ) \to
\hom_{\IZ}(K_k(EG \times_G X),\IZ)
\]
is injective and has finite cokernel.
Since $K_k(G\backslash X)$ is a finitely generated abelian group, we conclude
from~\eqref{Additivity_of_dim_p},~\eqref{dim_p(A)_for_A_fin.gen.} and the
Universal Coefficient Theorem for non-equivariant $K$-theory~\ref{the:Universal_coefficient_theorem_for_K-theory}
 applied to $G\backslash X$
\begin{eqnarray}
\dim\!\widehat{_p}(\hom_{\IZ}(K_k(EG \times_G X,\IZ)) & = & \dim_{\IZ}(K_k(G\backslash X))
\nonumber
\\ & = & \dim_{\IZ}(K^k(G\backslash X)).
\label{dim_p(hom(K_k(Gbackslash_X),Z)}
\end{eqnarray}
We conclude from Theorem~\ref{cor:UCT.for.Borel.theories}
and equations~\eqref{Additivity_of_dim_p},~\eqref{dim_p(ext(K_k,Z))} and~\eqref{dim_p(hom(K_k(Gbackslash_X),Z)}
\begin{eqnarray}
\dim\!\widehat{_p}(K^k(EG \times_G X)) & = & r^p_{k-1}(X) + \dim_{\IZ}(K^k(G\backslash X).
\label{dim_p(K_upper_k(EG_times_G_X):II}
\end{eqnarray}
Now~\eqref{dim_p(K_upper_k(EG_times_G_X):I} and~\eqref{dim_p(K_upper_k(EG_times_G_X):II}
imply $r^k_p(X) = r^p_{k-1}(X)$. This finishes the proof of
Lemma~\ref{lem:comparing_r_p_upper_kand_r_p_upper_for_K-theory}.
\end{proof}

\begin{lemma}\label{lem:sing_homology_iso_versus_generalized_(co)homology_iso}
  Let $f \colon X \to Y$ be a map of $CW$-complexes which induces isomorphisms
  $H_n(X;\IZ) \xrightarrow{\cong} H_n(Y;\IZ)$ for all $n \ge 0$. Then for any
  cohomology theory $\calh^*$ and any homology theory $\calh_*$ (satisfying the disjoint union axiom)
  and any $n \in  \IZ$,
  the maps $\calh^n(f) \colon \calh^n(Y) \to \calh^n(X)$ and $\calh_n(f)
  \colon \calh_n(X) \to \calh_n(Y)$ are isomorphisms.
\end{lemma}
\begin{proof}
Because of the excision and the suspension axiom for the cohomology theory $\calh^*$ if suffices to establish the result for the twofold suspension of $f$. Since the twofold suspension of a $CW$-complex is simply-connected it follows from the Hurewicz theorem that the twofold suspension of $f$ is a homotopy equivalence. The claim then follows from the homotopy invariance axiom for the cohomology theory $\calh^*$.
\end{proof}

Now we can give the proof of Theorem~\ref{the:intro}.

\begin{proof}[Proof of Theorem~\ref{the:intro}]
This follows from~Theorem~\ref{the:computing_calh_upper_ast(EG_times_G_X)_for_proper_X},
Theorem~\ref{the:numbers_r}, Theorem~\ref{the:computing_calh_lower_ast(EG_times_G_X)_for_proper_X},
Lemma~\ref{lem:comparing_r_p_upper_kand_r_p_upper_for_K-theory} and
Lemma~\ref{lem:sing_homology_iso_versus_generalized_(co)homology_iso}, as soon as we have shown
that for every  prime $p$, every  element $g \in G$ of $p$-power order and every $k \in \IZ$ we have an isomorphism
of $\IQ$-modules
\begin{eqnarray}
H^{k}(BC_G\langle g \rangle;\IQ) & \cong  & H^{k}(X^{\langle g \rangle}/C_G\langle g \rangle;\IQ).
\label{to_show}
\end{eqnarray}
We have $\widetilde{H}_k(X;\IZ) = 0$ for all $k \in \IZ$ by assumption.  Hence
we get $\widetilde{H}_k(X;\IF_p) = 0$ for all $k \in \IZ$ , where $\IF_p$ is the
field with $p$ elements.  By Smith theory (see~\cite[Theorem~5.2 in~III.5 on
page~130]{Bredon(1972)}) we conclude that $X^{\langle g   \rangle}$ is non-empty and
$\widetilde{H}_k(X^{\langle g   \rangle};\IF_p) = 0$ for all $k \in \IZ$. Hence
$\widetilde{H}_k(X^{\langle g  \rangle};\IQ) = 0$ for all $k \in \IZ$ and $\langle g \rangle$
is subconjugated to an isotropy group of $X$. Since $X$ is a finite $G$-$CW$-complex,
we conclude that  the sets $\con_p(G)$ and $\calp(G)$ are finite.

Let $Y$ be any proper $C_G\langle g \rangle$-$CW$-complex with
$\widetilde{H}_k(Y;\IQ) = 0$ for all $k \in \IZ$. Choose a $C_G\langle g
\rangle$-map $f \colon Y \to \underline{E}C_G\langle g \rangle$.  The
$\IQ[C_G\langle g \rangle]$-chain map
\[
C_*(f) \otimes_{\IZ} \id_{\IQ} \colon C_*(Y)
\otimes_{\IZ} \IQ \to C_*(\underline{E}C_G\langle g \rangle) \otimes_{\IZ} \IQ
\]
is a $\IQ[C_G\langle g \rangle]$-chain map of projective
$\IQ[C_G\langle g \rangle]$-chain complexes which induces an
isomorphism on homology. Hence it is a $\IQ[C_G\langle g
\rangle]$-chain homotopy equivalence. It induces a chain homotopy equivalence of
$\IQ$-chain complexes $C_*(f) \otimes_{\IZ[G\langle g \rangle]} \IQ$. Hence we
obtain a $\IQ$-isomorphism $H^k(Y/C_G\langle g \rangle;\IQ)
\xrightarrow{\cong} H^k(\underline{E}C_G\langle g \rangle)/C_G\langle g
\rangle;\IQ)$. If we apply this to $Y = X$ and $Y = EC_G\langle g \rangle$, we
get~\eqref{to_show}. This finishes the proof of Theorem~\ref{the:intro}.
\end{proof}

Next we deal with the equivariant universal coefficient theorem.
For a finite proper $G$-$CW$-complex its equivariant $K$-homology
$K^G_k(X)$ can be identified with the expression $K_k(\IC,C_0(X) \rtimes G)$
given by Kasparov's $KK$-theory.
The Kasparov intersection pairing yields for a finite proper
$G$-$CW$-complex a pairing
\begin{eqnarray}\label{special.KK.pairing}
 K^{k}_G(X) \otimes K_{k}^G(X) & \to & KK_0(\IC,\IC)\cong \IZ.
\end{eqnarray}
Taking adjoints gives homomorphisms
\begin{eqnarray}
\label{adjoint.before.UCT.1}
K_G^*(X) & \to & \hom_\IZ(K_*^G(X),\IZ);\\
\label{adjoint.before.UCT.2}
K^G_*(X) & \to & \hom_\IZ(K^*_G(X),\IZ).
\end{eqnarray}

Now we give the proof of the
Equivariant Universal Coefficient Theorem for $K$-theory~\ref{the:Equivariant_Universal_Coefficient_Theorem_for_K-theory}

\begin{proof}[Proof of Theorem~\ref{the:Equivariant_Universal_Coefficient_Theorem_for_K-theory}]
Assume first that $X$ is a proper orbit $G/H$. Green's imprimitivity
theorem \cite[\S 2]{Green(1978)} in this case says that $C_0(G/H)
\rtimes G$ and $\IC H$ are Morita equivalent, and therefore they are
also $KK$-equivalent. Hence $K^*_G(G/H) \cong K_*(C_0(G/H) \rtimes
G) \cong K_*(\IC H)$, and $K_*^G(G/H) \cong KK(C_0(G/H) \rtimes
G,\IC) \cong KK(\IC H,\IC)$. Since $\IC H$ is a finite $C^*$-algebra
$K^*_G(G/H)$ and $K_*^G(G/H)$ are both finitely generated
(projective) $\IZ$-modules it follows by induction over the
dimension and a subinduction over the number of cells of top
dimension that $K^*_G(X)$ and $K_*^G(X)$ are finitely generated for
all finite proper $G$-$CW$-complexes $X$. In particular (\ref{UCT1})
is exact if and only if (\ref{UCT2}) is.

By a result of Rosenberg and Schochet~\cite[Theorem~1.17 and Theorem~7.10]{Rosenberg-Schochet(1987)}
the sequence (\ref{UCT1}) is exact and splits unnaturally if $C_0(X) \rtimes G$
is contained in the category of $C^*$-algebras $\N$, introduced
in~\cite{Rosenberg-Schochet(1987)}. We shall prove that $C_0(X) \rtimes G\in\N$
by induction over the number of cells of $X$.

Since $C_0(G/H) \rtimes G$ is $KK$-equivalent to $\IC H$
and the category $\N$ contains all $C^*$-algebras which are
$KK$-equivalent to finite $C^*$-algebras
(c.f.~\cite[22.3.5 (b)]{Blackadar(1998)}) we conclude that
 the $C^*$-algebras $C_0(G/H) \rtimes G$ for all proper orbits
$G/H$ are in $\N$.  A second property of the category $\N$ is
that it is closed under extensions~\cite[22.3.4~(N3)]{Blackadar(1998)}.
If $X$ is a finite proper $G$-$CW$-complex for which $C_0(X) \rtimes G$ is in $\N$,
and $Y$ is obtained from $X$ by attaching a
proper $n$-dimensional $G$-cell $Z=G/H\times e_n$ we obtain
an exact sequence $0 \to C_0(Z) \to C_0(Y) \to C_0(X)\to  0$.
Taking the crossed product with $G$ we obtain an exact sequence
\[
0 \to C_0(Z)\rtimes G \to C_0(Y)\rtimes G \to C_0(X)\rtimes G\to 0.
\]
The extension property of $\N$ yields that $C_0(Y) \rtimes G$ is in $\N$.
\end{proof}

\begin{remark} \em \label{rem:Boekstedt_universal_coefficient_theorem}
Let $G$ be a finite group. The isomorphism of abelian groups
\[
\mu \colon R(G) \to  \hom_\IZ(R(G),\IZ)
\]
which sends $[V]$ to the map $\mu(V) \colon R(G) \to \IZ, [W]
\mapsto \dim_{\IC}(V \otimes_{\IC} W)^G),$
is the  special case $X = \pt$ of the
Equivariant Universal Coefficient Theorem for $K$-theory~%
\ref{the:Equivariant_Universal_Coefficient_Theorem_for_K-theory} It
is in fact an isomorphism of $R(G)$-modules and we get for any
$R_{\IC}(G)$-module $M$ a natural isomorphism of
$R_{\IC}(G)$-modules $\ext^i_{R_{\IC}(G)}(M,R_{\IC}(G))
\xrightarrow{\cong} \ext^i_{\IZ}(M,\IZ)$ for $i \ge 0$
(see~\cite[2.5 and 2.10]{Madsen(1986)}). Using change of ring
isomorphisms, the Equivariant
Universal Coefficient Theorem for $K$-theory~%
\ref{the:Equivariant_Universal_Coefficient_Theorem_for_K-theory}
is equivalent to the exactness of the short exact sequences
\begin{eqnarray*}
0 \to \ext_{R(G)}(K_{*-1}^G(X),R(G)) \to K^*_G(X) \to \hom_{R(G)}(K_*^G(X),R(G))\to 0,\\
0 \to \ext_{R(G)}(K^{*+1}_G(X),R(G)) \to K_*^G(X) \to \hom_{R(G)}(K^*_G(X),R(G))\to 0.
\end{eqnarray*}
The exactness of these sequences has been proved by B\"okstedt~\cite{Boekstedt(1981)}
using the concept of Anderson
duality. B\"okstedt's technique also can be used to prove the
Equivariant Universal Coefficient Theorem for $K$-theory~%
\ref{the:Equivariant_Universal_Coefficient_Theorem_for_K-theory}
for a finite group $G$.
\em
\end{remark}


\typeout{-----------------------  Section 6  ------------------------}

\section{Examples}
\label{sec:Examples}

\begin{example} \label{exa:SL_3(Z)} \em
Consider the group $G = SL_3(\IZ)$.
We conclude from~\cite[Corollary on page~8]{Soule(1978)}
that for $G = SL_3(\IZ)$ the quotient space $G\backslash\underline{E}G$ is contractible.
Hence the long exact sequence in
Theorem~\ref{the:computing_calh_upper_ast(EG_times_G_X)_for_proper_X}~%
\ref{the:computing_calh_upper_ast(EG_times_G_X)_for_proper_X:long_exact_sequence}
reduces to an isomorphism
\[
\widetilde{\calh}^k(BG) \xrightarrow{\cong} \prod_{p \in \calp(G)}
\widetilde{\calh}^k(BG;\IZ\widehat{_p})
\]
and the one of Theorem~\ref{the:computing_calh_lower_ast(EG_times_G_X)_for_proper_X}~%
\ref{the:computing_calh_lower_ast(EG_times_G_X)_for_proper_X:long_exact_sequence}
to the isomorphism
\[
\bigoplus_{p \in \calp(G)} \widetilde{\calh}_{k+1}(BG;\IZ/{p^\infty})
\xrightarrow{\cong} \widetilde{\calh}_k(BG).
\]
>From the classification of finite subgroups of $SL_3(\IZ)$ we
see that $SL_3(\IZ)$ contains up to conjugacy  four subgroups of order $2$
and two cyclic subgroups of order $3$. The cyclic subgroups of order $3$
have finite normalizers and the action of the normalizer on each of this group is
non-trivial. There are no cyclic subgroups of order $p$ for a prime
$p$ different from $2$ and $3$. Hence we see that $\con_2(G)$ contains four elements
and $\con_3(G)$ contains two elements. The rational homology of all the
centralizers of elements in $\con_2(G)$ and $\con_3(G)$ agree with the one of the trivial group
(see~\cite[Example~6.6]{Adem(1993b)}).
We get in the notation of Theorem~\ref{the:intro} that
$r_2^0(G) = 4$, $r_3^0(G) = 2$, $r_2^1(G) = 0$, $r_3^0(G) = 1$
and $r^k_p(G) = 0$ for $p \not= 2,3$ and all $k$.
We conclude from Theorem~\ref{the:intro} that there is an exact sequence
\[
0 \to \widetilde{K}^0(BG) \to (\IZ\widehat{_2})^4 \oplus (\IZ\widehat{_3})^2 \oplus B_0 \to C_0 \to 0
\]
and an isomorphism
\[
\widetilde{K}^1(BG) \cong D_1
\]
for finite abelian groups $B_0$, $C_0$ and $D_1$ which vanish after applying
$- \otimes_{\IZ} \IZ\left[\frac{1}{6}\right]$. Actually the computation using Brown-Petersen cohomology
and the Conner-Floyd relation in~\cite{Tezuka-Yagita(1992)} show that one can choose the groups
$B_0$, $C_0$ and $D_1$ to be zero.
\em
\end{example}

The next result  shall illustrate that the knowledge of the spaces
$\underline{E}G$ allows to reduce the computation of the (co-)homology of $BG$
to the one of its finite subgroups. Let $G$ be a discrete group.
Let $\calmfin$ be the subfamily of $\calfin$ consisting
of elements in $\calfin$ which are maximal in $\calfin$.
Consider the following assertions concerning $G$:
\begin{itemize}

\item[(M)] Every non-trivial finite subgroup of $G$ is contained in a unique maximal finite subgroup;

\item[(NM)] $M \in \calmfin \to N\!M = M$;

\end{itemize}

For such a group there is a  nice model for $\underline{E}G$ with as few non-free cells as possible.
Let $\{(M_i) \mid i \in I\}$ be the set of conjugacy classes of maximal
finite subgroups of $M_i \subseteq Q$. By attaching free $G$-cells we get
an inclusion of $G$-$CW$-complexes
$j_1 \colon \coprod_{i \in I} G \times_{M_i} EM_i \to EG$,
The  we obtain by~\cite[Corollary~2.11]{Lueck-Weiermann(2007)} a $G$-pushout
\begin{eqnarray}
& \xycomsquareminus{\coprod_{i \in I} G \times_{M_i} EM_i}{j_1}{EG}
{u_1}{f_1}{\coprod_{i \in I} G/M_i}{k_1}{\underline{E}G} &
\label{pushout_for_underline_EG}
\end{eqnarray}
where $u_1$ is the obvious $G$-map obtained by collapsing each $EM_i$ to a point.

Here are some examples of groups $Q$ which satisfy conditions (M) and (NM):
\begin{itemize}

\item Extensions $1 \to \IZ^n \to G \to F \to 1$ for finite $F$ such that the conjugation
  action of $F$ on $\IZ^n$ is free outside $0 \in \IZ^n$. \\[1mm]
The conditions (M), (NM)  are satisfied
by~\cite[Lemma 6.3]{Lueck-Stamm(2000)}. There are models for $\underline{E}G$
whose underlying space is $\IR^n$. The quotient $G\backslash\underline{E}G$ looks like the quotient
of $T^n$ by a finite group.

\item Fuchsian groups $F$ \\[1mm]
The conditions (M), (NM) are satisfied.
(see for instance~\cite[Lemma 4.5]{Lueck-Stamm(2000)}).
In~\cite{Lueck-Stamm(2000)} the larger class of cocompact planar
groups (sometimes also called cocompact NEC-groups) is treated.
The quotients $G\backslash \underline{E}G$ are closed orientable surfaces.

\item One-relator groups $G$\\[1mm]
Let $G$ be a one-relator group. Let $G = \langle (q_i)_{i \in I} \mid r \rangle$
be a presentation with one relation.  We only
have to consider the case, where $Q$ contains torsion.
Let $F$ be the free group with basis $\{q_i \mid i \in I\}$. Then $r$ is an element in
$F$. There exists an element $s \in F$ and an integer $m \ge 2$
such that $r = s^m$, the cyclic  subgroup $C$
generated by the class $\overline{s} \in Q$ represented
by $s$ has order $m$, any finite subgroup of $G$ is subconjugated to $C$ and
for any $q \in Q$ the implication
$q^{-1}Cq \cap C \not= 1 \Rightarrow q \in C$ holds.
These claims follows from~\cite[Propositions~5.17,~5.18 and~5.19 in~II.5 on pages~107 and~108]{Lyndon-Schupp(1977)}.
Hence $Q$ satisfies (M) and (NM). There are explicit two-dimensional models
for $\underline{E}G$ with one $0$-cell $G/C \times D^0 $, as many free $1$-cells $G \times D^1$  as there are elements in $I$ and
one free $2$-cell $G \times D^2$ (see~\cite[Exercise 2 (c) II. 5 on page~44]{Brown(1982)}).
\end{itemize}

\begin{theorem} \label{the:special_case_satisfying_(M)_and_(NM)}
Suppose that the discrete group $G$ satisfies conditions (M) and (NM).
Let $p \colon BG \to G\backslash \underline{E}G$ be the map induced by
the canonical $G$-map $EG \to \underline{E}G$ and $Bj_i \colon BM_i \to BG$ be the map induced by the inclusion
$j_i \colon M_i \to G$.  Then
\begin{enumerate}

\item \label{the:special_case_satisfying_(M)_and_(NM):cohomology}
Let $\calh^*$ be a cohomology theory satisfying the disjoint union axiom.
Then there is a long exact sequence
\begin{multline*}
\ldots \xrightarrow{\prod_{i \in I}  \widetilde{\calh}^{k-1}(Bj_i)}  \prod_{i \in I} \widetilde{\calh}^{k-1}(BM_i)
\xrightarrow{\delta^{k-1}}   \widetilde{\calh}^k(G\backslash \underline{E}G)
\\ \xrightarrow{\widetilde{\calh}^k(p)}  \widetilde{\calh}^k(BG)
\xrightarrow{\prod_{i \in I}  \widetilde{\calh}^k(Bj_i)} \prod_{i \in I}  \widetilde{\calh}^k(BM_i)
\xrightarrow{\delta^k} \widetilde{\calh}^{k+1}(G\backslash \underline{E}G) \to \ldots
\end{multline*}
The map $\widetilde{\calh}^k(p)$ is split injective after applying $- \otimes_{\IZ}
\IZ\left[\frac{1}{\calp(G)}\right]$, provided that $I$ is finite;

\item \label{the:special_case_satisfying_(M)_and_(NM):homology}
Let $\calh_*$ be a homology theory satisfying the disjoint union axiom.
Then there is a long exact sequence
\begin{multline*}
\ldots \xrightarrow{\widetilde{\calh}_{k+1}(p)} \widetilde{\calh}_{k+1}(G\backslash\underline{E}G)
\xrightarrow{\partial_{k+1}}
\bigoplus_{i \in I} \widetilde{\calh}_{k}(BM_i)
\xrightarrow{\bigoplus_{i \in I} \widetilde{\calh}_{k}(Bj_i)} \widetilde{\calh}_k(BG)
\\
\xrightarrow{\widetilde{\calh}_k(p)} \widetilde{\calh}_{k}(G\backslash\underline{E}G)
\xrightarrow{\partial_k}  \bigoplus_{i \in I}  \widetilde{\calh}_{k-1}(BM_i)
\xrightarrow{\bigoplus_{i \in I} \widetilde{\calh}_{k-1}(Bj_i)} \ldots
\end{multline*}
The map $\widetilde{\calh}_k(p)$ is split surjective after applying $- \otimes_{\IZ} \IZ\left[\frac{1}{\calp(G)}\right]$.

\end{enumerate}
\end{theorem}
\begin{proof}
These long exact sequences come from the Mayer-Vietoris sequences
associated to the pushout which is obtained from the
$G$-pushout~\eqref{pushout_for_underline_EG} by dividing out the $G$-action.
For the splitting after applying $- \otimes_{\IZ} \IZ\left[\frac{1}{\calp(G)}\right]$
see Lemma~\ref{lem:splitting_calhast(q)}
and its obvious homological version.
\end{proof}

\begin{example}\label{exa:Fuchsian_groups} \em
Let $F$ be a cocompact Fuchsian group with presentation
\[F= \langle a_1,b_1,\ldots,a_g,b_g,c_1,\ldots,c_t \mid
c_1^{\gamma_1}=\ldots =c_t^{\gamma_t}=c_1^{-1}\cdots c_t^{-1}[a_1,b_1]
\cdots [a_g,b_g] =1 \rangle
\]
for integers $g,t \ge 0$ and $\gamma_i > 1$. Then
$G\backslash \underline{E}G$ is an orientable closed surface $S_g$
of genus $g$. Since $S_g$ is stably a wedge of spheres, we have
\[
\calh^n(S_g) \cong
\calh^n(\pt) \oplus \calh^{n - 1}(\pt)^{2g}\oplus \calh^{n-2}(\pt).
\]
If we suppose that $\calh^n(\pt)$ is torsionfree for all $n \in \IZ$, then we obtain from
Theorem~\ref{the:special_case_satisfying_(M)_and_(NM)}~\ref{the:special_case_satisfying_(M)_and_(NM):cohomology}
for  every $n \in \IZ$ the exact sequence
\[
0 \to \calh^n(S_g) \to \widetilde{\calh}^n(BG)
\to  \prod_{i=1}^t \widetilde{\calh}^n(B\IZ/\gamma_i) \to 0.
\]
\em
\end{example}



\begin{thebibliography}{10}

\bibitem{Abels(1978)}
H.~Abels.
\newblock A universal proper ${G}$-space.
\newblock {\em Math. Z.}, 159(2):143--158, 1978.

\bibitem{Adams(1969b)}
J.~F. Adams.
\newblock Lectures on generalised cohomology.
\newblock In {\em Category Theory, Homology Theory and their Applications, III
  (Battelle Institute Conference, Seattle, Wash., 1968, Vol. Three)}, pages
  1--138. Springer, Berlin, 1969.

\bibitem{Adem(1993b)}
A.~Adem.
\newblock Characters and ${K}$-theory of discrete groups.
\newblock {\em Invent. Math.}, 114(3):489--514, 1993.

\bibitem{Anderson(1964)}
D.~Anderson.
\newblock Universal coefficient theorems for $k$-theory.
\newblock mimeographed notes, Berkeley.

\bibitem{Artin-Mazur(1969)}
M.~Artin and B.~Mazur.
\newblock {\em Etale homotopy}.
\newblock Lecture Notes in Mathematics, No. 100. Springer-Verlag, Berlin, 1969.

\bibitem{Atiyah-McDonald(1969)}
M.~F. Atiyah and I.~G. Macdonald.
\newblock {\em Introduction to commutative algebra}.
\newblock Addison-Wesley Publishing Co., Reading, Mass.-London-Don Mills, Ont.,
  1969.

\bibitem{Atiyah-Segal(1969)}
M.~F. Atiyah and G.~B. Segal.
\newblock Equivariant ${K}$-theory and completion.
\newblock {\em J. Differential Geometry}, 3:1--18, 1969.

\bibitem{Baum-Connes-Higson(1994)}
P.~Baum, A.~Connes, and N.~Higson.
\newblock Classifying space for proper actions and ${K}$-theory of group
  ${C}\sp \ast$-algebras.
\newblock In {\em $C\sp \ast$-algebras: 1943--1993 (San Antonio, TX, 1993)},
  pages 240--291. Amer. Math. Soc., Providence, RI, 1994.

\bibitem{Blackadar(1998)}
B.~Blackadar.
\newblock {\em {$K$}-theory for operator algebras}, volume~5 of {\em
  Mathematical Sciences Research Institute Publications}.
\newblock Cambridge University Press, Cambridge, second edition, 1998.

\bibitem{Boekstedt(1981)}
M.~Boekstedt.
\newblock Universal coefficient theorems for equivariant {$K$}-and
  {$KO$}-theory.
\newblock Aarhus Preprint series, 1981/82 No. 7, 1981.

\bibitem{Borel-Serre(1973)}
A.~Borel and J.-P. Serre.
\newblock Corners and arithmetic groups.
\newblock {\em Comment. Math. Helv.}, 48:436--491, 1973.
\newblock Avec un appendice: Arrondissement des vari\'et\'es \`a coins, par A.
  Douady et L. H\'erault.

\bibitem{Bredon(1972)}
G.~E. Bredon.
\newblock {\em Introduction to compact transformation groups}.
\newblock Academic Press, New York, 1972.
\newblock Pure and Applied Mathematics, Vol. 46.

\bibitem{Brown(1982)}
K.~S. Brown.
\newblock {\em Cohomology of groups}, volume~87 of {\em Graduate Texts in
  Mathematics}.
\newblock Springer-Verlag, New York, 1982.

\bibitem{Green(1978)}
P.~Green.
\newblock {The local structure of twisted covariance algebras.}
\newblock {\em Acta Math.}, 140:191--250, 1978.

\bibitem{Greenlees(1993a)}
J.~P.~C. Greenlees.
\newblock ${K}$-homology of universal spaces and local cohomology of the
  representation ring.
\newblock {\em Topology}, 32(2):295--308, 1993.

\bibitem{Grothendieck(1957)}
A.~Grothendieck.
\newblock {Sur quelques points d'alg\`ebre homologique.}
\newblock {\em Tohoku Math. J., II. Ser.}, 9:119--221, 1957.

\bibitem{Leary-Nucinkis(2001a)}
I.~J. Leary and B.~E.~A. Nucinkis.
\newblock Every {CW}-complex is a classifying space for proper bundles.
\newblock {\em Topology}, 40(3):539--550, 2001.

\bibitem{Lueck(1989)}
W.~L{\"u}ck.
\newblock {\em Transformation groups and algebraic ${K}$-theory}, volume 1408
  of {\em Lecture Notes in Mathematics}.
\newblock Springer-Verlag, Berlin, 1989.

\bibitem{Lueck(2005s)}
W.~L{\"u}ck.
\newblock Survey on classifying spaces for families of subgroups.
\newblock In {\em Infinite groups: geometric, combinatorial and dynamical
  aspects}, volume 248 of {\em Progr. Math.}, pages 269--322. Birkh\"auser,
  Basel, 2005.

\bibitem{Lueck(2007)}
W.~L{\"u}ck.
\newblock Rational computations of the topological {$K$}-theory of classifying
  spaces of discrete groups.
\newblock {\em J. Reine Angew. Math.}, 611:163--187, 2007.

\bibitem{Lueck-Oliver(2001b)}
W.~L{\"u}ck and B.~Oliver.
\newblock Chern characters for the equivariant ${K}$-theory of proper
  ${G}$-{C}{W}-complexes.
\newblock In {\em Cohomological methods in homotopy theory (Bellaterra, 1998)},
  pages 217--247. Birkh\"auser, Basel, 2001.

\bibitem{Lueck-Oliver(2001a)}
W.~L{\"u}ck and B.~Oliver.
\newblock The completion theorem in ${K}$-theory for proper actions of a
  discrete group.
\newblock {\em Topology}, 40(3):585--616, 2001.

\bibitem{Lueck-Stamm(2000)}
W.~L{\"u}ck and R.~Stamm.
\newblock Computations of ${K}$- and ${L}$-theory of cocompact planar groups.
\newblock {\em $K$-Theory}, 21(3):249--292, 2000.

\bibitem{Lueck-Weiermann(2007)}
W.~L\"uck and M.~Weiermann.
\newblock On the classifying space of the family of virtually cyclic subgroups.
\newblock Preprintreihe SFB 478 --- Geometrische Strukturen in der Mathematik,
  Heft 453, M\"unster, arXiv:math.AT/0702646v2, to appear in the Proceedings in
  honour of Farrell and Jones in Pure and Applied Mathematic Quarterly, 2007.

\bibitem{Lyndon-Schupp(1977)}
R.~C. Lyndon and P.~E. Schupp.
\newblock {\em Combinatorial group theory}.
\newblock Springer-Verlag, Berlin, 1977.
\newblock Ergebnisse der Mathematik und ihrer Grenzgebiete, Band 89.

\bibitem{Madsen(1986)}
I.~Madsen.
\newblock Geometric equivariant bordism and {$K$}-theory.
\newblock {\em Topology}, 25(2):217--227, 1986.

\bibitem{Meintrup(2000)}
D.~Meintrup.
\newblock {\em On the Type of the Universal Space for a Family of Subgroups}.
\newblock PhD thesis, Westf\"alische Wilhelms-Universit\"at M\"unster, 2000.

\bibitem{Meintrup-Schick(2002)}
D.~Meintrup and T.~Schick.
\newblock A model for the universal space for proper actions of a hyperbolic
  group.
\newblock {\em New York J. Math.}, 8:1--7 (electronic), 2002.

\bibitem{Mislin(2010)}
G.~Mislin.
\newblock Classifying spaces for proper actions of mapping class groups.
\newblock {\em M\"unster J. of Mathematics}, 3:263--272, 2010.

\bibitem{Phillips(1989)}
N.~C. Phillips.
\newblock {\em Equivariant ${K}$-theory for proper actions}.
\newblock Longman Scientific \& Technical, Harlow, 1989.

\bibitem{Rosenberg-Schochet(1987)}
J.~Rosenberg and C.~Schochet.
\newblock The {K}\"unneth theorem and the universal coefficient theorem for
  {K}asparov's generalized {$K$}-functor.
\newblock {\em Duke Math. J.}, 55(2):431--474, 1987.

\bibitem{Serre(1979)}
J.-P. Serre.
\newblock Arithmetic groups.
\newblock In {\em Homological group theory (Proc. Sympos., Durham, 1977)},
  volume~36 of {\em London Math. Soc. Lecture Note Ser.}, pages 105--136.
  Cambridge Univ. Press, Cambridge, 1979.

\bibitem{Soule(1978)}
C.~Soul{\'e}.
\newblock The cohomology of {${\rm SL}\sb{3}({\bf Z})$}.
\newblock {\em Topology}, 17(1):1--22, 1978.

\bibitem{Switzer(1975)}
R.~M. Switzer.
\newblock {\em Algebraic topology---homotopy and homology}.
\newblock Springer-Verlag, New York, 1975.
\newblock Die Grundlehren der mathematischen Wissenschaften, Band 212.

\bibitem{Tezuka-Yagita(1992)}
M.~Tezuka and N.~Yagita.
\newblock Complex ${K}$-theory of ${B}{\rm {s}{l}}\sb 3({\mathbf {z}})$.
\newblock {\em $K$-Theory}, 6(1):87--95, 1992.

\bibitem{Dieck(1987)}
T.~tom Dieck.
\newblock {\em Transformation groups}.
\newblock Walter de Gruyter \& Co., Berlin, 1987.

\bibitem{Weibel(1994)}
C.~A. Weibel.
\newblock {\em An introduction to homological algebra}.
\newblock Cambridge University Press, Cambridge, 1994.

\bibitem{Whitehead(1978)}
G.~W. Whitehead.
\newblock {\em Elements of homotopy theory}, volume~61 of {\em Graduate Texts
  in Mathematics}.
\newblock Springer-Verlag, New York, 1978.

\bibitem{Yoneda(1960)}
N.~Yoneda.
\newblock {On ext and exact sequences.}
\newblock {\em J. Fac. Sci., Univ. Tokyo, Sect. I}, 8:507--576, 1960.

\bibitem{Yosimura(1975)}
Z.-i. Yosimura.
\newblock Universal coefficient sequences for cohomology theories of {${\rm
  CW}$}-spectra.
\newblock {\em Osaka J. Math.}, 12(2):305--323, 1975.

\end{thebibliography}

\end{document}